\documentclass[11pt,a4paper]{amsart}

\usepackage[T1]{fontenc}
\usepackage{enumitem}
\usepackage{lmodern}
\usepackage{microtype}
\usepackage[textwidth=160mm,textheight=238mm,centering]{geometry}
\usepackage{amssymb}
\usepackage{xcolor}
\usepackage[colorlinks=true,linkcolor=blue!45!black,citecolor=blue!45!black,urlcolor=blue!45!black]{hyperref}
\hypersetup{
  pdftitle={Dimension formula for harmonic measures in invariant random graphs on Gromov hyperbolic spaces},
  pdfauthor={Kohki Sakamoto}
}
\newtheorem{theorem}{Theorem}[section]
\newtheorem*{theorem*}{Theorem}
\newtheorem{lemma}[theorem]{Lemma}
\newtheorem*{lemma*}{Lemma}
\newtheorem{proposition}[theorem]{Proposition}
\newtheorem*{proposition*}{Proposition}
\newtheorem{corollary}[theorem]{Corollary}
\newtheorem*{corollary*}{Corollary}

\newtheorem*{claim*}{Claim}

\newtheorem*{fact*}{Fact}

\newtheorem*{conjecture*}{Conjecture}
\theoremstyle{definition}
\newtheorem{definition}[theorem]{Definition}
\newtheorem*{definition*}{Definition}
\newtheorem{example}[theorem]{Example}
\newtheorem*{example*}{Example}
\newtheorem{remark}[theorem]{Remark}
\newtheorem*{remark*}{Remark}
\newtheorem{question}[theorem]{Question}
\newtheorem*{question*}{Question}

\newtheorem*{assumption*}{Assumption}
\numberwithin{equation}{section}
\allowdisplaybreaks[1]

\DeclareMathOperator{\id}{id}
\DeclareMathOperator{\bnd}{bnd}

\newcommand{\R}{{\mathbb R}}

\newcommand{\Z}{{\mathbb Z}}
\newcommand{\N}{{\mathbb N}}

\address{Graduate School of Mathematical Sciences, the University of Tokyo, 3-8-1 Komaba, Meguro, Tokyo 153-8914, Japan.}
\email{sakamoto-kohki571@g.ecc.u-tokyo.ac.jp}
\begin{document}
\title[Dimension formula for invariant random graphs]{Dimension formula for harmonic measures in invariant random graphs on Gromov hyperbolic spaces}
\author{Kohki Sakamoto}
\date{}

\begin{abstract}

For random walks in various group-invariant random graphs on Gromov hyperbolic spaces, we show that the associated harmonic measures are exact dimensional and their Hausdorff dimensions are equal to the entropy over the speed. Our method is inspired by cluster relations introduced by Gaboriau and applies to percolation clusters on hyperbolic groups and Poisson--Voronoi tessellations on rank one Riemannian symmetric spaces.

\end{abstract}
\maketitle

\section{Introduction}

\subsection{Background}
A random walk on a proper Gromov hyperbolic space determines the hitting measure on the Gromov boundary under some mild conditions (cf. \cite{MR1815698}). This hitting measure is called the \textit{harmonic measure} of the random walk. One of the main subjects in the study of such harmonic measures is to establish dimensional properties called exact dimensionality and the dimension formula. Let us explain these terms. Let $(X, d_{X})$ be a proper Gromov hyperbolic space, and $\partial X$ be its Gromov boundary. For a Borel measure $\nu$ on the boundary $\partial X$, its \textit{upper Hausdorff dimension} is defined by 
\[
\dim \nu = \underset{\xi \in \partial X}{\nu\text{-}\sup} \liminf_{r \to 0} \frac{\log \nu(B_{\rho}(\xi, r))}{\log r},
\]
where ``$\nu\text{-}\sup$'' indicates the essential supremum with respect to $\nu$, and $B_{\rho}(\xi, r)$ denotes the ball centered at $\xi$ with radius $r$ with respect to a natural quasi-metric $\rho$ on $\partial X$. Informally speaking, it shows the degree of the fractalness of $\nu$. We say that a harmonic measure $\nu$ is \textit{exact dimensional} if
\[
\dim \nu = \lim_{r \to 0} \frac{\log \nu(B_{\rho}(\xi, r))}{\log r}  
\]
holds for $\nu$-almost every $\xi \in \partial X$. Further, we say that $\nu$ satisfies the \textit{dimension formula} if $\dim \nu$ is equal to $h/l$, where $h$ and $l$ are the entropy and the speed of the random walk that determines $\nu$, respectively (see Section 2 for the definition). The most extensively studied case is when $X$ is the Cayley graph of a Gromov hyperbolic group $\Gamma$. It has been proved that the harmonic measure of a random walk driven by a single measure on $\Gamma$ is exact dimensional and satisfies the dimension formula, under several types of conditions (\cite{MR1832436}, \cite{MR2286061}, \cite{MR2919980}, \cite{MR3893268}). Such results can be used to show that the harmonic measure is singular to natural geometric measures on the boundary like Patterson-Sullivan measures (\cite{MR2286061}, \cite{MR2919980}, \cite{MR4243517}). 

On the other hand, there are only a few works concerning harmonic measures associated with random walks in random graphs or random environments. Lyons, Pemantle, and Peres showed that the harmonic measure of the simple random walk on a supercritical Galton-Watson tree is exact dimensional and satisfies the dimension formula in their influential paper \cite{MR1336708}. The notable work that inspires our study is the paper by Kaimanovich \cite{MR1631732}, where he established the exact dimensionality and the dimension formula for a large class of Markov chains on trees, and it can be applied to random walks in stationary random environments on free groups. 

\subsection{Invariant percolations}
Our first result is a generalization of such results in the setting of random environments on groups. Namely, we establish exact dimensionality and the dimension formula for random walks in percolation clusters and random conductance models on hyperbolic groups. Here we focus on percolation clusters for simplicity. Let $G$ be a Cayley graph of a nonelementary hyperbolic group $\Gamma$. For $p \in [0, 1]$, the \textit{Bernoulli percolation} $\mu_{p}$ is defined as the random subgraph of $G$ obtained by independently retaining or deleting each edge with probability $p$ or $1-p$, respectively. The Bernoulli percolation $\mu_{p}$ is called \textit{supercritical} if $\mu_{p}$-almost every $\omega \subset G$ has a connected component with infinitely many vertices. For supercritical $\mu_{p}$, we consider the simple random walk on the connected component of $\omega$ containing the identity $1$ (denoted by $C_{\omega}(1)$), conditioned on $C_{\omega}(1)$ being infinite. Note that the event of $C_{\omega}(1)$ being infinite, denoted by $\Omega_{1}$, is $\mu_{p}$-positive since $\mu_{p}$ is supercritical and $\Gamma$-invariant. Benjamini, Lyons, and Schramm initiated the study of such random walks in percolation clusters on general Cayley graphs in \cite{MR1802426}. They proved that the entropy and the speed (denoted by $h$ and $l$, respectively) of such random walks are deterministic, i.e., they depend only on $p$. Further, they showed that $h, l >0$ if the group is nonamenable. Therefore, for supercritical $\mu_{p}$ on $G$, we can define the harmonic measure associated with the simple random walk on $C_{\omega}(1)$ for $\mu_{p}$-almost every $\omega \in \Omega_{1}$ (recall that every nonelementary hyperbolic group is nonamenable). Based on the ergodic-theoretic approach in \cite{MR1802426}, we establish the exact dimensionality and the dimension formula for the simple random walks in supercritical Bernoulli percolation clusters. 

\begin{theorem}\label{intromainBernolli}
    Let $\mu_{p}$ be a supercritical Bernoulli percolation on $G$. Then, for $\mu_{p}$-almost every $\omega \in \Omega_{1}$, letting $\nu_{\omega}$ be the harmonic measure on $\partial G$ determined by the simple random walk on $C_{\omega}(1)$ starting from 1, we have
    \[
    \dim \nu_{\omega} = \lim_{r \to 0} \frac{\log \nu_{\omega}(B_{\rho}(\xi, r))}{\log r} = \frac{h}{l}
    \]
    for $\nu_{\omega}$-almost every $\xi \in \partial G$. In particular, $\dim \nu_{\omega}$ is positive and constant for $\mu_{p}$-almost every $\omega \in \Omega_{1}$.
\end{theorem}

In fact, our strategy works for a more general class of percolation models on $G$, as follows. 

\begin{theorem}\label{intromainpercpositive}
   Let $\mu$ be an ergodic $\Gamma$-invariant percolation on $G$ having indistinguishable infinite clusters. Assume that the simple random walk on $C_{\omega}(1)$ has positive speed for $\mu$-almost every $\omega \in \Omega_{1}$. Then, for $\mu$-almost every $\omega \in \Omega_{1}$, letting $\nu_{\omega}$ be the harmonic measure on $\partial G$ determined by the simple random walk on $C_{\omega}(1)$ starting from 1, we have
    \[
    \dim \nu_{\omega} = \lim_{r \to 0} \frac{\log \nu_{\omega}(B_{\rho}(\xi, r))}{\log r} = \frac{h}{l}
    \]
    for $\nu_{\omega}$-almost every $\xi \in \partial G$. In particular, $\dim \nu_{\omega}$ is positive and constant for $\mu$-almost every $\omega \in \Omega_{1}$.
\end{theorem}

We refer to Section 5 for the precise definitions of the terms appearing in this theorem. Note that Benjamini, Lyons, and Schramm gave some sufficient conditions for positive speed in \cite{MR1802426}, and hence Theorem \ref{intromainpercpositive} can be applied to various models (see Theorem \ref{BLS}). For example, every ergodic invariant percolation with a unique infinite cluster satisfies the assumptions in our theorem. 

\subsection{Invariant Voronoi tessellations} Although all the above results are about discrete subgroups or random environments on them, our method is applicable to spatial random graphs on non-discrete Gromov hyperbolic spaces like the hyperbolic space $\mathbb{H}^{d}$. Namely, we establish exact dimensionality and the dimension formula for the simple random walk on Voronoi tessellations on rank one Riemannian symmetric spaces of noncompact type constructed from \textit{invariant point processes.} 

Let us explain these terms and some background on random Voronoi tessellations. Let $X = G / K$ be a Riemannian symmetric space. A \textit{$G$-invariant point process} on $X$ is a random discrete subset of $X$ whose distribution is invariant under the $G$-action. For example, the Poisson point process on $X$ with a $G$-invariant intensity measure is $G$-invariant. Using a $G$-invariant point process on $X$, we can construct random geometric objects, and they serve as natural random discretizations of the symmetric space. In particular, \textit{Voronoi tessellations} and \textit{Delaunay graphs} (defined as the duals of Voronoi tessellations) have attracted much attention in connection with various fields such as network theory, material science, and biology. For a more detailed account of the history and applications, we refer to \cite{MR1295245} and references therein.  

Although most studies concerning random Voronoi tessellations and Delaunay graphs have been carried out over Euclidean spaces, in their seminal paper \cite{MR1815220}, Benjamini and Schramm studied Bernoulli percolation on Voronoi tessellations and Delaunay graphs constructed from Poisson point processes in the hyperbolic plane, and proved that the situation is completely different from the Euclidean setting. Since then, probability on Poisson--Voronoi tessellations in hyperbolic spaces has been investigated intensively. Among them, Benjamini, Paquette, and Pfeffer \cite{MR3813982} proved that the simple random walk on Poisson-Delaunay graph in the hyperbolic plane has \textit{positive speed}, i.e., the hyperbolic distance from the origin grows linearly with time. Further, this result was extended to invariant Voronoi tessellations on Riemannian symmetric spaces of noncompact type (\cite{MR3880013}, \cite{MR4341072}). When the ambient space is Gromov hyperbolic (for example, when $X = \mathbb{H}^{d}$), the positivity of the speed allows us to consider the hitting measure on the boundary determined by the walk. For the Poisson process on $\mathbb{H}^d$, Carrasco, Lessa, and Paquette already proved that the speed, entropy, and Hausdorff dimension of the quenched harmonic measure are deterministic, and that $\dim\nu_{\omega}\le h/l$ (Section 12.1, Lemma 14 in \cite{carrasco2017furstenbergtypeformulaspeed}). Our second main result is exact dimensionality and the dimension formula for invariant point processes under the following assumptions, giving the matching lower bound in the Poisson case.

\begin{theorem}\label{mainthmrankone}
    Let $X = G / K$ be a rank one Riemannian symmetric space of noncompact type. Let $\mu$ be an ergodic $G$-invariant point process of finite positive intensity on $X$, and let $\mu_o$ denote its Palm measure. Assume the following conditions hold:
    \begin{itemize}
        \item[(i)] $\mathbb{E}_{\mu_{o}}[\deg_{\mathcal{G}_{\omega}}(o)\log\deg_{\mathcal{G}_{\omega}}(o)] < \infty$.
        \item[(ii)] $\mathbb{E}_{\mu_{o}}[{\sum_{x \in \omega, (o,x) \in E(\mathcal{G}_{\omega})} d_{X}(o,x)}] < \infty$. 
        \item[(iii)] There exists $\alpha > 0$ such that 
        \[
        \sum_{n \in \Z_{+}} \mu_{o}\{ \omega \in K \backslash M_{o}(X) \colon \# \omega \cap B_{X}(o, n) > \alpha^{n} \} < \infty.
        \]
    \end{itemize} Then, for $\mu_{o}$-almost every $\omega$, the harmonic measure $\nu_{\omega}$ determined by the simple random walk on the Delaunay graph $\mathcal{G}_{\omega}$ starting at $o$ is exact dimensional and satisfies
    \[
    \dim \nu_{\omega} = \frac{h}{l},
    \]
    where $h$ and $l$ denote the entropy and the speed of the walk, respectively.
\end{theorem}

Since it is straightforward to check that Poisson point processes on $X$ with respect to a $G$-invariant volume satisfy the conditions in the above theorem, we have the following corollary.

\begin{theorem}
    Let $X = \mathbb{H}^{d}$ be the $d$-dimensional hyperbolic space with $d \ge 2$, and $\mu_{o}$ be the Palm measure of the Poisson point process on $X$ whose intensity measure is an isometry-invariant volume. Then, for $\mu_{o}$-almost every $\omega$, the harmonic measure determined by the simple random walk on the Delaunay graph $\mathcal{G}_{\omega}$ is exact dimensional and satisfies 
    \[
    \dim \nu_{\omega} = \frac{h}{l},
    \]
    where $h$ and $l$ denote the entropy and the speed of the simple random walk on $\mathcal{G}_{\omega}$, respectively.
\end{theorem}

\subsection{Outline of the proof}
Let us give an overview of the proof of the main results, focusing on the setting of invariant percolations. First, the upper bound,
\[
\limsup_{r \to 0} \frac{\log \nu_{\omega}(B_{\rho}(\xi, r))}{\log r} \le \frac{h}{l}
\]
for $\nu_{\omega}$-almost every $\xi$, follows from essentially the same argument as in \cite{MR2286061}. The difficulty arises in the proof of the lower bound, i.e.,
\[
\liminf_{r \to 0} \frac{\log \nu_{\omega}(B_{\rho}(\xi, r))}{\log r} \ge \frac{h}{l}
\]
for $\nu_{\omega}$-almost every $\xi$.
We first prove that, for every $\epsilon > 0$, the set $D_{\epsilon}$ defined by
\[
D_{\epsilon} = \biggl\{ (\omega, \xi) \in \Omega_{1} \times \partial G  \colon  \liminf_{r \to 0} \frac{\log \nu_{\omega}(B_{\rho}(\xi, r))}{\log r} \ge \frac{h}{l} - \epsilon \biggr\}
\]
is positive with respect to the measure on $\Omega_{1} \times \partial G$ given by $\int_{\Omega_{1}} \delta_{\omega} \otimes \nu_{\omega} d\mu(\omega)$. This part follows Tanaka's proof of Theorem 3.3 in \cite{MR3893268}, where a similar claim was shown for random walks driven by a single measure. He proved that the harmonic measure determined by the random walk driven by a single measure with finite first moment is exact dimensional and satisfies the dimension formula. In his proof, the ergodicity of the $\Gamma$-action on $\partial G$ was used to prove that the subset of $\partial G$ defined analogously to $D_{\epsilon}$ is conull, and it completes the proof of the lower bound in his setting. 

In our setting, to prove that $D_{\epsilon}$ is conull, the first attempt should be to consider the diagonal action on $\Omega_{1} \times \partial G$, instead of the boundary action. However, since $\Omega_{1}$ is not $\Gamma$-invariant, we cannot do this naively. Instead, we define a subrelation of the orbit equivalence relation of the diagonal action, inspired by cluster relations in \cite{MR2221157}, so that we can prove the ergodicity of the relation and show that $D_{\epsilon}$ is invariant under the relation. Combining the ergodicity and the invariance with the positivity of $D_{\epsilon}$, we complete the proof. This argument is different in nature from Kaimanovich's approach for free groups \cite{MR1631732}.

Our ergodicity argument also applies to invariant Voronoi tessellations, though the quotient by the compact subgroup $K$ necessitates further care. We carry this out following the recent approach of Ab\'ert and Mellick \cite{MR4555892}.

\subsection{Related works and open questions}
In the study of percolation on hyperbolic spaces, regular tilings of the hyperbolic plane $\mathbb{H}^2$ have served as a fundamental model. Percolation on such tilings and their dual graphs has been extensively studied. Carrasco, Lessa, and Paquette \cite{MR4243517} provided an explicit lower bound for the speed, which, via the upper bound for the dimension of the harmonic measure, implies a dimension drop. More generally, it has been shown in various settings that the dimension of the harmonic measure associated with a random walk is strictly smaller than the dimension of the boundary (see e.g., \cite{MR1336708, MR1832436, MR3763407, MR4230410, MR4514005}). In the context of random walks in percolation clusters, one can consider a stronger version of this phenomenon:
\begin{question}Let $\mu$ be a $\Gamma$-invariant percolation on $G$ as in Theorem \ref{mainPerc}. When is the inequality
\[
\dim \nu_{\omega} < \dim \Lambda(C_{\omega}(1))
\]
strict for $\mu$-almost every $\omega$? Here $\Lambda(C_{\omega}(1))$ denotes the \textit{limit set} of $C_{\omega}(1)$, i.e., $\overline{C_{\omega}(1)}^{G \cup \partial G} \backslash G$.

\end{question}
A similar problem was posed by Benjamini, Paquette, and Pfeffer \cite{MR3813982} for Poisson--Voronoi tessellations of $\mathbb{H}^d$. Specifically, they questioned whether the harmonic measure of the simple random walk is singular with respect to the Lebesgue measure class on the boundary. 
\begin{question}
Consider the harmonic measure $\nu_{\omega}$ determined by the simple random walk on an invariant Voronoi tessellation in $\mathbb{H}^d$. Is the measure $\nu_{\omega}$ always singular with respect to the Lebesgue measure class on $S^{d-1}$?\end{question}
Furthermore, we consider the behavior of the dimension of harmonic measures under parameter variation.
\begin{question}
Let $\mu_{p}$ be a supercritical Bernoulli percolation on $G$ and $\delta_{p}$ denote the dimension of the associated harmonic measures. Is the function $p \mapsto \delta_{p}$ continuous or monotonically increasing?
\end{question}

\begin{question}
Let $X$ be a rank-one Riemannian symmetric space of non-compact type and $\delta_{\lambda}$ be the dimension of the harmonic measures determined by the Voronoi tessellation with intensity $\lambda$. Is the function $\lambda \mapsto \delta_{\lambda}$ continuous or monotonically decreasing?
\end{question}

Several remarks are in order regarding the regularity of these dimensions. First, by our Theorem \ref{mainPerc}, the continuity of $\delta_p$ can be reduced to that of the entropy and the speed. In the related setting of hyperbolic branching Brownian motion, there is a known discrepancy between the regularity of the limit set and the harmonic measure. While the dimension of the limit set is discontinuous in the branching parameter \cite{MR1452555}, Geldbach \cite{geldbach2025hyperbolicbranchingbrownianmotion} recently showed that the dimension of the hitting measure is a linearly increasing function of the rate. This suggests the possibility of a stronger regularity for harmonic measures. Regarding the limit sets of percolation clusters, the situation is more subtle. For regular tilings of $\mathbb{H}^2$, Lalley \cite{MR1873136} proved that the dimension of the limit set is continuous in $p$. Although no explicit examples of discontinuity are currently known for percolation limit sets, recent work by Hutchcroft and Pan \cite{hutchcroft2024dimensionjumpuniquenessthreshold} suggests the existence of hyperbolic groups where such discontinuity might occur. We believe that these developments highlight the importance of clarifying the regularity of $\delta_p$ and $\delta_\lambda$.

\subsection{Organization of the paper}
In Section 2, we review some definitions and basic properties concerning Gromov hyperbolic spaces, random walks on them, and Hausdorff dimensions of measures. In Section 3, we show two estimates for general Markov chains on hyperbolic groups with positive speed. In Section 4, we first treat invariant random conductance models on groups with the uniform ellipticity condition, where the arguments are simpler than those for invariant percolations. In Section 5, we develop ergodic theory of invariant percolations and use it to prove the exact dimensionality and the dimension formula. In Section 6, we start with the definitions and basic properties about invariant point processes. Then, we develop ergodic theory of invariant point processes and prove exact dimensionality and the dimension formula for random walks on invariant Voronoi tessellations.

\subsection*{Acknowledgements}
The author would like to thank Yoshikata Kida and Ryokichi Tanaka for their support and helpful comments. The author is also grateful to the referee for a careful reading and many helpful corrections and suggestions. This research is supported by JSPS KAKENHI Grant number 25KJ0862 and FoPM, WINGS Program, the University of Tokyo.

\section{Preliminaries}

\subsection{Geometry and random walks of Gromov hyperbolic spaces}

In this section, we briefly review the basics of hyperbolic spaces. Let us start with the definition of hyperbolicity in the sense of Gromov \cite{MR0919829}.

\begin{definition}[$\delta$-hyperbolicity]
    Let $(X, d_{X})$ be a metric space. For $x, y, z \in X$, we define the \textit{Gromov product} of $x, y$ over $z$ by
    \[
    (x\,|\,y)_{z} = \frac{d_{X}(x,z) + d_{X}(y,z) - d_{X}(x,y)}{2}.
    \]
    Let $\delta$ be a non-negative number. We say that $(X, d_{X})$ is $\delta$-\textit{hyperbolic} if 
    \[
    (x\,|\,y)_{w} \ge \min \{(x\,|\,z)_{w}, (y\,|\,z)_{w} \} - \delta
    \]
    for all $x, y, z, w \in X$. We say that $(X, d_{X})$ is \textit{hyperbolic} if it is $\delta$-hyperbolic for some $\delta \ge 0$.
\end{definition}

\begin{example}
   The hyperbolic space $\mathbb{H}^{d}$ is Gromov hyperbolic with respect to the invariant metric $d_{X}$. More generally, a rank one symmetric space of noncompact type is Gromov hyperbolic with respect to the invariant metric.
\end{example}

\begin{definition}[Hyperbolic groups]
    Let $\Gamma$ be a finitely generated group. We say that $\Gamma$ is \textit{hyperbolic} if there exists a finite generating set $S$ of $\Gamma$ such that the Cayley graph $G$ associated with $(\Gamma, S)$ is hyperbolic with respect to the graph metric $d_{G}$. A hyperbolic group $\Gamma$ is called \textit{elementary} if it is finite or virtually $\Z$. We always take $1$ as the base point of $G$ and define $\lvert g\rvert = d_{G}(1, g)$ for $g \in \Gamma$.  
\end{definition}

Next, we define the Gromov boundary of a Gromov hyperbolic space. We fix a base point $o \in X$.  For $x, y \in X$, $(x\,|\,y)$ denotes the Gromov product of $x$ and $y$ over the base point $o$.

\begin{definition}[Gromov boundary]
    Let $(X, d_{X})$ be a proper geodesic Gromov hyperbolic space. We define the \textit{boundary} $\partial X$ as follows:
    \begin{itemize}
        \item The set $\partial X$ is the quotient of the set of geodesic rays starting at $o$ by identifying two rays when they are within a bounded distance. 
        \item The quasi-metric $\rho$ on $\partial X$ is defined by $\rho(\xi, \xi')= \exp{(-(\xi \, | \, \xi'))}$, where $(\xi \,|\,\xi')$ is defined by 
        \[
        (\xi\,|\,\xi') = \sup_{\gamma\in[\xi],\,\gamma'\in[\xi']} \liminf_{n,m\to\infty} (\gamma(n)\,|\,\gamma'(m)). 
        \]
        Here, $[\xi]$ denotes the set of unit-speed geodesic rays starting at $o$ representing $\xi$.
    \end{itemize}
    Further, every isometry on $X$ acts on $\partial X$ as a homeomorphism, in the natural way.
\end{definition}

\begin{remark}
Although $\rho$ is not a genuine metric, for every small $\epsilon$, there exists a genuine metric $\rho_{\epsilon}$ satisfying
\[
C_{\epsilon}^{-1}\rho(\xi, \xi')^{\epsilon} \le \rho_{\epsilon}(\xi, \xi') \le C_{\epsilon} \rho(\xi, \xi')^{\epsilon}
\]
where $C_{\epsilon}$ is a constant depending only on $\epsilon$.
The fact that $\rho$ is not a genuine metric has no effect on the argument concerning Hausdorff dimensions. Then we only treat $\rho$ for the sake of simplicity.  
\end{remark}

For the boundary action of an isometry of $X$, we have the following estimate.

\begin{lemma}[Lemma 2.2 in \cite{MR3893268}]\label{gammaball}
    For an isometry $g$ of $X$, there exists $c_{g} >0$ such that 
    \[
    B_{\rho}(g\xi, c_{g}^{-1}R) \subset gB_{\rho}(\xi, R) \subset B_{\rho}(g\xi, c_{g}R)
    \]
    for every $R>0$ and $\xi \in \partial X$.
\end{lemma}

Let us introduce the shadow of $x \in X$ on the boundary.

\begin{definition}
    Let $x \in X$ and $R>0$. We define the \textit{shadow} of $(x, R)$ by
    \[
    S(x, R) = \{ \xi \in \partial X \colon (\xi|x) > d_{X}(o, x) - R \}.
    \]
\end{definition}

We have the following comparison between shadows and balls.

\begin{lemma}[Proposition 2.1 in \cite{MR2919980}]\label{shadowball}
    There exist $C>0$ and $R_{0}>0$ such that 
    \[
    B_{\rho}(\xi, C^{-1} e^{-d_{X}(o, x) + R}) \subset S(x, R) \subset B_{\rho}(\xi, C e^{-d_{X}(o, x) + R})
    \]
    whenever $R \ge R_{0}$ and $x \in X$ is on a geodesic ray representing $\xi$.
\end{lemma}

The following result, often called \textit{geodesic tracking property}, is fundamental in the study of random walks in Gromov hyperbolic spaces.

\begin{theorem}[Section 7 in \cite{MR1815698}]\label{geodesictracking}
    Let $x = (x_{n})_{n \in \Z_{+}}$ be a sequence in a proper geodesic Gromov hyperbolic space $(X, d_X)$. Assume that it satisfies the following two conditions:
    \begin{itemize}
        \item $\lim_{n \to \infty} \frac{1}{n}d_{X}(x_{n}, x_{n+1}) = 0$.
        \item $\lim_{n \to \infty} \frac{1}{n} d_X(x_{0}, x_{n}) = l$ for some strictly positive number $l$, which is called the speed of the sequence.
    \end{itemize}
    Then, there exists a unique element $\xi_{x} \in \partial X$ such that 
    \[
    \lim_{n \to \infty} \frac{1}{n}d_X(x_{n}, \xi_{x}(ln)) = 0.
    \]
\end{theorem}

This theorem allows us to define the harmonic measure for a random walk with positive speed. 

\begin{definition}
    Let $X$ be a proper geodesic Gromov hyperbolic space, and $\mathbb{P}$ be a Borel probability measure on $X^{\Z_{+}}$ such that $\mathbb{P}$-almost every $(x_{n})_{n \in \Z_{+}}$ satisfies the conditions in Theorem \ref{geodesictracking}. Then, by Theorem \ref{geodesictracking}, we can define the \textit{boundary map} $x \mapsto \xi_{x}$ for $\mathbb{P}$-almost every $x \in X^{\Z_{+}}$. This map is denoted by $\bnd$, and the image measure $\bnd_{*}\mathbb{P}$ is called the \textit{harmonic measure}. Note that the map $\bnd$ is $\text{Isom}(X)$-equivariant.
\end{definition}

Next, we introduce the notion of regularity of a random walk. This terminology is due to Kaimanovich \cite{MR1631732}.

\begin{definition}
    Let $(X, d_{X})$ be a metric space and $\mathbb{P}$ be the distribution of a Markov chain on $X$. We say that the walk is \textit{regular} if there exist real numbers $l$ and $h$ such that for $\mathbb{P}$-almost every $(x_{n})_{n \in \Z_{+}}$, the following hold:
    \begin{itemize}
        \item $\lim_{n \to \infty} \frac{1}{n} d_{X}(x_{0}, x_{n})  = l$.
        \item $\lim_{n \to \infty} -\frac{1}{n} \log \mathbb{P}(\{(X_{m})_{m \in \Z_{+}} \in X^{\Z_{+}} \colon X_{n} = x_{n}\}) = h$.
        \item $\lim_{n \to \infty} \frac{1}{n}d_{X}(x_{n}, x_{n+1}) = 0$.
    \end{itemize}
    Note that the last condition has nothing to do with $l$ or $h$. The limits $l$ and $h$ are called the \textit{speed} and the \textit{entropy} of the walk, respectively.
\end{definition}

\subsection{Hausdorff dimensions of measures}

\begin{definition}[Hausdorff dimensions of measures]
   Let $\nu$ be a Borel probability measure on $\partial X$. We define the \textit{(upper) Hausdorff dimension} of $\nu$ by
   \[
   \dim \nu = \inf \{ \dim_{\rho} E \colon \nu(E) = 1 \},
   \]
   where $\dim_{\rho}E$ denotes the Hausdorff dimension of $E$ with respect to the quasi-metric $\rho$. 
\end{definition}

We have the following characterization of the dimension.

\begin{lemma}\label{Frostman}
  In the above setting, we have
  \[
  \dim \nu = \underset{\xi \in \partial X}{\nu\text{-}\sup} \liminf_{r \to 0} \frac{\log \nu (B_{\rho}(\xi, r))}{\log r},  
  \]
  where ``$\nu\text{-}\sup$'' indicates the essential supremum with respect to $\nu$, and $B_{\rho}(\xi, r)$ is the ball of radius $r$ centered at $\xi$ with respect to the quasi-metric $\rho$.
\end{lemma}

For the proof, we refer to Section 8.7 in \cite{MR1800917}.
 
\begin{definition}[Exact dimensionality]
    Let $\nu$ be a Borel probability measure on $\partial X$. We say that $\nu$ is \textit{exact dimensional} if the limit
    \[
    \lim_{r \to 0} \frac{\log \nu (B_{\rho}(\xi, r))}{\log r}
    \]
    exists for $\nu$-almost every $\xi$ and it is constant $\nu$-almost everywhere. Note that if $\nu$ is exact dimensional then the above limit is equal to $\dim \nu$ by Lemma \ref{Frostman}. 
\end{definition}

\section{General Estimates for Dimensions}

The goal of this section is to prove the following two general estimates about regular Markov chains on hyperbolic groups. Let $\Gamma$ denote a nonelementary $\delta$-hyperbolic group and $G$ denote its Cayley graph.

\begin{theorem}\label{genupperbound}
     Let $X$ be a regular and nearest neighbor Markov chain having positive speed on $G$. Then, for $\nu$-almost every $\xi$, we have
    \[
    \limsup_{r \to 0} \frac{\log \nu(B_{\rho}(\xi, r))}{\log r} \le \frac{h}{l},
    \]
    where $h$ and $l$ denote the entropy and the speed of $X$, respectively. In particular, $\dim \nu \le \frac{h}{l}$.
\end{theorem}

\begin{theorem}\label{genlowerbound}
    Let $X$ be a regular and nearest neighbor Markov chain having positive speed on $G$. Then, for every $\epsilon > 0$, the following subset $D_{\epsilon}$ is $\nu$-positive:
    \[
    D_{\epsilon} := \biggl\{ \xi \in \partial G : \liminf_{r \to 0} \frac{\log \nu(B_{\rho}(\xi, r))}{\log r} \ge \frac{h}{l} - \epsilon \biggr\}.
    \]
\end{theorem}

\begin{proof}[Proof of Theorem \ref{genupperbound}.]
    We follow Section 2 in \cite{MR2286061}. Let $\mathbb{P}$ denote the probability on $\Gamma^{\Z_{+}}$ defined as the distribution of $X$. For $0<\epsilon<\min\{l,1/2\}$ and $N \in \N$, we define the set 
    \[
    A_{\epsilon, N} = \{ x \in \Gamma^{\Z_{+}} \colon (x_{n+1} \,|\, x_{n}) \ge (l-\epsilon)n \, \, \text{and} \, \, p^{n}(1, x_{n}) \ge e^{-(h+\epsilon)n} \, \, \text{for every} \, \, n \ge N \}.
    \]
    Then, for each $\epsilon$, there exists $N_{\epsilon}$ such that $\mathbb{P}(A_{\epsilon, N_{\epsilon}}) \ge 1-\epsilon$ by the regularity of $X$. Let $A_{\epsilon}$ denote this $A_{\epsilon, N_{\epsilon}}$. We first show that the limit
    \[
    \lim_{n \to \infty} \mathbb{P}(A_{\epsilon}\,|\,\{x_n=z_n\})
    \]
    exists for almost every $z \in A_{\epsilon}$. This follows from the Markov property of $X$. Indeed, for $n>N_{\epsilon}$, we have
    \[
    \lim_{n \to \infty} \mathbb{P}(A_{\epsilon}\,|\,\{x_n=z_n\})
    =
    \lim_{n \to \infty} \mathbb{P}(A_{\epsilon}\,|\,\{x_m=z_m \text{ for every }m\ge n\})
    =
    \mathbb{P}(A_{\epsilon}\,|\,\operatorname{tail}(z))
    \]
    by the martingale convergence theorem, where $\operatorname{tail}$ denotes the projection to the tail boundary. We define
    \[
    A'_{\epsilon}=\{z \in A_{\epsilon}: \mathbb{P}(A_{\epsilon}\,|\,\operatorname{tail}(z))>\epsilon\}.
    \]
    Then, we have
    \begin{align*}
    1-\epsilon
    &\le \mathbb{P}(A_{\epsilon})
    =\int \mathbb{P}(A_{\epsilon}\,|\,\operatorname{tail}(x))\,d\mathbb{P}(x)\\
    &=\int_{A'_{\epsilon}} \mathbb{P}(A_{\epsilon}\,|\,\operatorname{tail}(x))\,d\mathbb{P}(x)
    +\int_{(A'_{\epsilon})^c} \mathbb{P}(A_{\epsilon}\,|\,\operatorname{tail}(x))\,d\mathbb{P}(x)\\
    &\le \mathbb{P}(A'_{\epsilon})+\epsilon,
    \end{align*}
    and hence $\mathbb{P}(A'_{\epsilon})\ge1-2\epsilon$. Recall that for $z \in A_{\epsilon}$, there exists $C > 0$ such that 
    \[
    A_{\epsilon} \cap \{x_{n} = z_{n}\} \subset \{ x \in \Gamma^{\Z_{+}} \colon \xi_{x} \in B_{\rho}(\xi_{z}, Ce^{-(l-\epsilon)n})\}
    \]
    for every $n \ge N_{\epsilon}$. Then, for $z \in A'_{\epsilon}$, we have 
    \begin{align*}
        \limsup_{r \to 0} \frac{\log \nu(B_{\rho}(\xi_{z}, r))}{\log r} 
        &= \limsup_{n \to \infty} \frac{\log \nu(B_{\rho}(\xi_{z}, Ce^{-(l-\epsilon)n}))}{-(l-\epsilon)n} \\
        &\le \limsup_{n \to \infty} \frac{\log \mathbb{P}(A_{\epsilon} \cap \{x_{n} = z_{n}\})}{-(l-\epsilon)n} \\
        &= \limsup_{n \to \infty} \frac{\log \mathbb{P}(\{x_{n}=z_{n}\})}{-(l-\epsilon)n}
        = \limsup_{n \to \infty} \frac{\log p^{n}(1, z_{n})}{-(l-\epsilon)n} \\
        &\le \frac{h+\epsilon}{l-\epsilon}.
    \end{align*}
    Since $\epsilon$ can be taken arbitrarily and $\mathbb{P}(A'_{\epsilon})\ge1-2\epsilon$, this implies  
    \[
    \limsup_{r \to 0} \frac{\log \nu(B_{\rho}(\xi, r))}{\log r} \le \frac{h}{l}
    \]
    for $\nu$-almost every $\xi$.
    We also have $\dim \nu \le \frac{h}{l}$ by Lemma \ref{Frostman}. 
\end{proof}

Next, we give the proof of Theorem \ref{genlowerbound}, following Theorem 3.3 in \cite{MR3893268}. In the proof, we consider the conditional measures associated with the factor map $\bnd \colon (\Gamma^{\Z_{+}}, \mathbb{P}) \to (\partial G, \nu)$. More explicitly, that is a family of probability measures $(\mathbb{P}^{\xi})_{\xi \in \partial G}$ on $\Gamma^{\Z_{+}}$ such that
  \[
  \mathbb{P} = \int_{\partial G} \mathbb{P}^{\xi} \, \, d\nu.
  \]

\begin{proof}[Proof of Theorem \ref{genlowerbound}.]
    It is enough to prove that for every $\epsilon>0$ there exists a $\nu$-positive set $F_{\epsilon}$ such that
    \[
    \liminf_{r \to 0} \frac{\log \nu(F_{\epsilon} \cap B_{\rho}(\xi, r))}{\log r} \ge \frac{h}{l} - \epsilon 
    \]
    for $\nu$-almost every $\xi \in \partial G$. Indeed, for such $F_{\epsilon}$, we have $\dim \nu \ge \dim \nu|_{F_{\epsilon}} \ge \frac{h}{l} - \epsilon$ by the definition of Hausdorff dimensions and Lemma \ref{Frostman}. Letting $\epsilon\to0$, we obtain $\dim\nu\ge h/l$. By Lemma \ref{Frostman}, this implies that $D_{\epsilon}$ is positive for every $\epsilon>0$. In the rest of the proof, we construct such $F_{\epsilon}$. Fix $0<\epsilon<1/2$ and define $A_{\epsilon, N}$ as follows:
    \[
    A_{\epsilon, N} = \{ x \in \Gamma^{\Z_{+}} \colon d(x_{n}, \xi_{x}(ln)) \le \epsilon n \, \, \text{and} \, \, p^{n}(1, x_{n}) \le e^{-(h-\epsilon)n} \, \, \text{for every} \, \, n \ge N \}
    \]
    for $N \in \N$. Then there exists $N_{\epsilon}$ such that $\mathbb{P}(A_{\epsilon, N_{\epsilon}}) \ge 1-\epsilon$ by regularity of $X$ and Theorem \ref{geodesictracking}. Let $A_{\epsilon} = A_{\epsilon, N_{\epsilon}}$ and $F_{\epsilon} = \{ \xi \, \colon \, \mathbb{P}^{\xi}(A_{\epsilon}) \ge \epsilon \}$. Then we can estimate the size of $F_{\epsilon}$ as follows: 
    \begin{align*}
    1- \epsilon
    &\le \mathbb{P}(A_{\epsilon}) = \int \mathbb{P}^{\xi}(A_{\epsilon}) \, d\nu(\xi) \\
    &= \int_{F_{\epsilon}} \, \mathbb{P}^{\xi}(A_{\epsilon}) \, d\nu(\xi) + \int_{F^{c}_{\epsilon}} \, \mathbb{P}^{\xi}(A_{\epsilon}) \, d\nu(\xi) \\
    &\le \nu(F_{\epsilon}) + \epsilon
    \end{align*}
    and hence $\nu(F_{\epsilon}) \ge 1-2\epsilon$. Let $N \ge N_{\epsilon}$ and $R > \max\{4\delta, R_{0}\}$, where $R_{0}$ is the constant in Lemma \ref{shadowball}. For $z \in A_{\epsilon, N}$, we define $y_{n}$ as $\xi_{z}(ln)$. Then, by Lemma 3.6 in \cite{MR3893268} (note that the assumption $R > 4\delta$ is used here), there exists $C > 0$ such that for all $n \ge N$ and $x \in A_{\epsilon}$, $x_{n}$ belongs to the ball $B(y_{n}, 2\epsilon n + C)$ if $\xi_{x} \in S(y_{n}, R)$. Therefore, we have 
    \begin{align*}
    \mathbb{P}(\xi_{x} \in F_{\epsilon} \cap S(y_{n}, R)) 
    &\le \mathbb{P}(A_{\epsilon} \cap \{ x_{n} \in B(y_{n}, 2\epsilon n + C) \}) + \mathbb{P}(A^{c}_{\epsilon} \cap \{ \xi_{x} \in F_{\epsilon} \cap S(y_{n}, R) \}) \\
    &\le \mathbb{P}(A_{\epsilon} \cap \{ x_{n} \in B(y_{n}, 2\epsilon n + C) \}) + (1-\epsilon)\nu(F_{\epsilon} \cap S(y_{n}, R))
    \end{align*}
    for every $n \ge N$.
    The last inequality follows from the construction of $F_{\epsilon}$. Then we have 
    \begin{align*}
    \epsilon \nu(F_{\epsilon} \cap S(y_{n}, R)) 
    &\le \mathbb{P}(A_{\epsilon} \cap \{ x_{n} \in B(y_{n}, 2\epsilon n + C) \}) \\
    &\le d^{2\epsilon n + C} \times e^{-(h-\epsilon)n},
    \end{align*}
    where $d$ is the degree of $G$, after increasing $C$ if necessary. 
    From this inequality, we obtain
    \begin{align*}
    \liminf_{n \to \infty} \frac{\log \nu(F_{\epsilon} \cap S(y_{n}, R))}{-ln}
    \ge \liminf_{n \to \infty} \frac{-\log \epsilon -(h-\epsilon)n + (2\epsilon n + C)\log d}{-ln}
    = \frac{h-(1+2\log d)\epsilon}{l}
    \end{align*}
    Further, by applying Lemma \ref{shadowball} to $y_{n}$, which is on a geodesic ray from $1$ to $\xi_{z}$, we have 
    \[
    B_{\rho}(\xi_{z}, C^{-1}e^{-ln+R}) \subset S(y_{n}, R)
    \]
    and hence
    \begin{align*}
        \liminf_{r \to 0} \frac{\log \nu(F_{\epsilon} \cap B_{\rho}(\xi_{z}, r))}{\log r}
        &= \liminf_{n \to \infty} \frac{\log \nu(F_{\epsilon} \cap B_{\rho}(\xi_{z}, C^{-1}e^{-ln + R}))}{\log (C^{-1}e^{-ln + R})} \\
        &= \liminf_{n \to \infty} \frac{\log \nu(F_{\epsilon} \cap B_{\rho}(\xi_{z}, C^{-1}e^{-ln + R}))}{-ln} \\
        &\ge \liminf_{n \to \infty} \frac{\log \nu(F_{\epsilon} \cap S(y_{n}, R))}{-ln}
    \end{align*}
    since $C$ and $R$ are positive constants. 

Finally, for $z \in A_{\epsilon, N}$, we obtain 
    \begin{align*}
        \liminf_{r \to 0} \frac{\log \nu(F_{\epsilon} \cap B_{\rho}(\xi_{z}, r))}{\log r} 
        &\ge \liminf_{n \to \infty} \frac{\log \nu(F_{\epsilon} \cap S(y_{n}, R))}{-ln} \\
        &\ge \frac{h}{l} - \frac{(1 + 2\log d)\epsilon}{l}
    \end{align*}
    from the above estimates.
    Since the sequence $(A_{\epsilon, N})_{N \ge N_{\epsilon}}$ gives an exhaustion of $\Gamma^{\Z_{+}}$ modulo null sets, we obtain the desired property after replacing $\epsilon$ by a sufficiently small positive multiple of $\epsilon$.
\end{proof}

\section{Invariant Random Conductance Models}

In this section, we consider uniformly elliptic random conductance models. They can be seen as bounded random perturbations of $G$, and the proof of the exact dimensionality and the dimension formula is simpler than the case of percolation clusters.

\subsection{Ergodic theory of random conductance models}

\begin{definition}[$\Gamma$-invariant random conductance models]
Let $\mathcal{C}$ be the \textit{space of conductances} $\R_{>0}^{E}$, where $E$ denotes the edge set of $G$. Note that $\Gamma$ acts on $\mathcal{C}$ by translations. A $\Gamma$-\textit{invariant random conductance model} is a $\Gamma$-invariant probability measure on $\mathcal{C}$. We say that a $\Gamma$-invariant random conductance model $\mu$ is \textit{ergodic} if it is ergodic under the $\Gamma$-action. Each element $\omega$ of $\mathcal{C}$ determines a nearest neighbor Markov kernel on $G$ as follows:
\[
p_{\omega}(g, gs) = \frac{\omega(g,gs)}{\sum_{s' \in S} \omega(g,gs')}
\]
for $g \in \Gamma$ and $s \in S$, and other transition probabilities are zero. This kernel is denoted by $P_{\omega}$ and the corresponding Markov chain starting from $1$ is denoted by $X_{\omega}$.  
We call $\mu$ \textit{uniformly elliptic} if $\text{supp}(\mu)$ is contained in $(\alpha, 1/\alpha)^{E}$ for some $\alpha \in (0, 1)$. 
\end{definition}

\begin{definition}[Path bundles]
Let $\mu$ be a $\Gamma$-invariant uniformly elliptic random conductance model on $G$. We define the weighted version $\mu'$ of $\mu$ by 
\[
    d\mu'(\omega)= \frac{\sum_{s \in S} \omega(1,s)}{D} d\mu(\omega),
\]
where $D$ is the expectation of the sum of conductances at $1$.
We define the \textit{path bundle} $\Pi$ by $\mathcal{C} \times \Gamma^{\Z_{+}}$ as a Borel space, and we endow $\Pi$ with the probability measure $\lambda$ defined by 
\[
  \lambda = \int_{\mathcal{C}} \delta_{\omega} \otimes \mathbb{P}_{\omega} d\mu'(\omega),
\]
where $\mathbb{P}_{\omega}$ denotes the distribution of sample paths of $X_{\omega}$, which is a probability measure on $\Gamma^{\Z_{+}}$.
Further, we define the \text{shift map} $T$ on $\Pi$ by $(\omega, (x_{n})_{n \in \Z_{+}}) \mapsto (x_{1}^{-1}\omega, (x_{1}^{-1}x_{n+1})_{n \in \Z_{+}})$.
\end{definition} 

The following result is crucial for our ergodic-theoretic approach. Note that this type of result is quite standard in the context of random walks in random environments. 

\begin{theorem}\label{fundthmRCM}
In the above setting, assume that $\mu$ is ergodic. Then the system $(\Pi, \lambda, T)$ is an ergodic probability-measure-preserving system.  \end{theorem}

Note that this type of result is well known in the context of random walks in random environments (see \cite{MR2278453}, for example). Since the proof is similar for all three cases (random conductance models, percolations, and Voronoi tessellations), we refer to Proposition \ref{pmperg} (in the case of point processes).

This theorem can be used to deduce that $X_{\omega}$ is regular for $\mu$-almost every $\omega$. Further, we can show that the entropy and the speed are positive and constant for $\mu$-almost every $\omega$.

\begin{theorem}\label{entspRCM}
Let $\mu$ be an ergodic $\Gamma$-invariant uniformly elliptic random conductance model. Then for $\mu$-almost every $\omega \in \mathcal{C}$, the random walk $X_{\omega}$ is regular, and its entropy and speed do not depend on $\omega$ and are positive. 
\end{theorem}
    
\begin{proof}
 First, we show that for $\mu$-almost every $\omega$, $X_{\omega}$ is regular and its entropy and speed are constant. For $n \in \Z_{+}$, we define functions on $\Pi$ as follows:
 \begin{itemize}
     \item $\phi_{n}(\omega, x) = \lvert x_{n} \rvert$,
     \item $\psi_{n}(\omega, x) = -\log p^{n}_{\omega}(1, x_{n})$.
 \end{itemize}
 Then we have 
\begin{align*}
    \phi_{m+n}(\omega, x) &= \lvert x_{m+n} \rvert \le \lvert x_{m} \rvert + \lvert x_{m}^{-1}x_{m+n} \rvert = \phi_{m}(\omega, x) + \phi_{n}(T^{m}(\omega, x))
\end{align*}
and 
\begin{align*}
     \psi_{m+n}(\omega, x) 
     &= -\log p^{m+n}_{\omega}(1, x_{m+n}) \\
     &\le -\log p^{m}_{\omega}(1, x_{m}) p^{n}_{\omega}(x_{m}, x_{m+n}) \\
     &= -\log p^{m}_{\omega}(1, x_{m}) - \log p^{n}_{x_{m}^{-1}\omega}(1, x_{m}^{-1}x_{m+n}) \\
     &= \psi_{m}(\omega, x) + \psi_{n}(T^{m}(\omega, x)).
\end{align*}  
Therefore, we can apply Kingman's ergodic theorem to $\phi_{n}$ and $\psi_{n}$, and hence for $\mu$-almost every $\omega$, $X_{\omega}$ is regular and its entropy and speed are constant. We show that they are positive in the rest of the proof. It is known that the entropy is positive if and only if the speed is positive (see Proposition 3.6 in \cite{MR2994841} for example). Therefore it is enough to show that the speed is positive. For a subset $B$ of the vertices of $G$, we define its edge-boundary $\partial_{E}(B)$ by 
\[
\partial_{E}(B) = \{ e \in E \colon e \, \, \text{connects a vertex in} \, \, B \, \, \text{and a vertex in} \, \, B^{c} \}.
\]
Then, for $\omega \in \operatorname{supp}(\mu)$ and every nonempty finite subset $K$ of the vertices of $G$, we have
\[
\frac{\sum_{e \in \partial_{E}(K)} \omega(e)}{\lvert K \rvert} \ge \alpha \frac{\lvert \partial_{E}(K) \rvert}{\lvert K \rvert} 
\]
where $\alpha > 0$ is the lower bound for $\omega(e)$, guaranteed by the uniform ellipticity. Since $G$ is nonamenable, this implies that the network determined by $\omega$ is nonamenable. Then, the claim follows from the fact that the random walk on a nonamenable network with exponential growth (note that $G$ has exponential growth since $\Gamma$ is nonamenable) has positive speed (see Section 6.2 in \cite{MR3616205}). 
\end{proof}

Then, we can consider the harmonic measures (Theorem \ref{geodesictracking}) and the ray bundle in the above setting, as follows.

\begin{definition}[Ray bundles]
    The \textit{ray bundle} $\Xi$ is $\mathcal{C} \times \partial G$ as a Borel space, and we consider the probability measure $\eta$ on $\Xi$ by
    \[
    \eta = \int_{\mathcal{C}} \delta_{\omega} \otimes \nu_{\omega} d\mu'(\omega),
    \]
    where $\nu_{\omega}$ denotes the harmonic measure associated with $X_{\omega}$ for $\mu$-almost every $\omega$. Note that $\eta$ is equal to the pushforward measure $(\id \times \bnd)_{*}\lambda$.
\end{definition}

\begin{lemma}\label{diagerg}
    The diagonal action of $\Gamma$ on $\Xi = \mathcal{C} \times \partial G$ is ergodic with respect to $\eta$.
\end{lemma}

\begin{proof}
    Let $A$ be a $\Gamma$-invariant subset of $\Xi$. Since $\eta$ is equal to the pushforward measure $(\id \times \bnd)_{*}\lambda$, it is enough to prove that the inverse image of $A$ under $\id \times \bnd$ is $\lambda$-null or $\lambda$-conull. Note that the $\Gamma$-invariance of $A$ implies the $T$-invariance of $(\id \times \bnd)^{-1}(A)$ by the definition of $T$. Since $T$ is ergodic with respect to $\lambda$ by Theorem \ref{fundthmRCM}, it implies the claim. 
\end{proof}

\subsection{The dimension formula for invariant random conductance models}

In this subsection, we prove the main theorem for random conductance models as above. First, we have the following two results:

\begin{theorem}\label{RCMupperbound}
    Let $\mu$ be an ergodic $\Gamma$-invariant uniformly elliptic random conductance model on $G$. Then, for $\mu$-almost every $\omega \in \mathcal{C}$ and $\nu_{\omega}$-almost every $\xi$, we have
    \[
    \limsup_{r \to 0} \frac{\log \nu_{\omega}(B_{\rho}(\xi, r))}{\log r} \le \frac{h}{l},
    \]
    where $h$ and $l$ denote the entropy and the speed of $X_{\omega}$, respectively. In particular, $\dim \nu_{\omega} \le \frac{h}{l}$.
\end{theorem}

\begin{proof}
    This follows from Theorem \ref{genupperbound} and Theorem \ref{entspRCM}.
\end{proof}

\begin{theorem}\label{RCMlowerbound}
    Let $\mu$ be an ergodic $\Gamma$-invariant uniformly elliptic random conductance model on $G$. Then, for every $\epsilon > 0$, the set $D_{\epsilon}$ defined by
    \[
    D_{\epsilon} = \biggl\{ (\omega, \xi) \in \mathcal{C} \times \partial G : \liminf_{r \to 0} \frac{\log \nu_{\omega}(B_{\rho}(\xi, r))}{\log r} \ge \frac{h}{l} - \epsilon \biggr\}
    \]
    is $\eta$-positive.
\end{theorem}
 
\begin{proof}
 This follows from Theorem \ref{genlowerbound} and Theorem \ref{entspRCM}.     
\end{proof}

Combining these theorems with the ergodicity established in the previous section, we have the following result.

\begin{theorem}\label{mainRCM}
    Let $\mu$ be an ergodic $\Gamma$-invariant uniformly elliptic random conductance model on $G$. Then, for $\mu$-almost every $\omega \in \mathcal{C}$, letting $\nu_{\omega}$ be the harmonic measure on $\partial G$ determined by $X_{\omega}$, we have
    \[
    \dim \nu_{\omega} = \lim_{r \to 0} \frac{\log \nu_{\omega}(B_{\rho}(\xi, r))}{\log r} = \frac{h}{l}
    \]
    for $\nu_{\omega}$-almost every $\xi$, where $h$ and $l$ denote the entropy and the speed of $X_{\omega}$, respectively. In particular, $\dim \nu_{\omega}$ is positive and constant for $\mu$-almost every $\omega \in \mathcal{C}$.
\end{theorem}

\begin{proof}
     We have already shown the upper bound in Theorem \ref{RCMupperbound}. The lower bound can be deduced from Theorem \ref{RCMlowerbound} as follows. We want to show that $D_{\epsilon}$ is $\eta$-conull for every $\epsilon > 0$. It is enough to prove that $D_{\epsilon}$ is invariant under the diagonal action of $\Gamma$ since the diagonal action of $\Gamma$ is ergodic with respect to $\eta$ (Lemma \ref{diagerg}) and $D_{\epsilon}$ is $\eta$-positive (Theorem \ref{RCMlowerbound}). Let $(\omega, \xi) \in D_{\epsilon}$ and $g \in \Gamma$. For $g \in \Gamma$, let $\mathbb{P}_{\omega, g}$ denote the distribution of sample paths starting from $g$ following $P_{\omega}$. Then, for every Borel subset $A \subset \partial G$, we have 
    \begin{align*}
    (g_{*}\nu_{\omega})(A) 
    &= \mathbb{P}_{\omega}(\{ x \in \Gamma^{\Z_{+}} \colon \xi_{x} \in g^{-1}A\}) \ge \mathbb{P}_{\omega}(\bigcup_{n \in \Z_{+}} \{x_{n} = g^{-1} \} \cap \{ x \in \Gamma^{\Z_{+}} \colon \xi_{x} \in g^{-1}A\}) \\
    &= \mathbb{P}_{\omega}(\bigcup_{n \in \Z_{+}} \{x_{n} = g^{-1} \}) \mathbb{P}_{\omega, g^{-1}}(\{ x \in \Gamma^{\Z_{+}} \colon \xi_{x} \in g^{-1}A\}) \\
    &=  \mathbb{P}_{\omega}(\bigcup_{n \in \Z_{+}} \{x_{n} = g^{-1} \}) \mathbb{P}_{g\omega}(\{ x \in \Gamma^{\Z_{+}} \colon \xi_{x} \in A \})  
    = \mathbb{P}_{\omega}(\bigcup_{n \in \Z_{+}} \{x_{n} = g^{-1} \}) \nu_{g\omega}(A),
    \end{align*}
    where we have used the strong Markov property of $X_{\omega}$ in the second equality. The third equality follows from the equality $\mathbb{P}_{g\omega} = g_{*}\mathbb{P}_{\omega, g^{-1}}$. Further, by Lemma \ref{gammaball}, we have $\nu_{g\omega}(B_{\rho}(g\xi,r)) 
    \le \nu_{g\omega}(gB_{\rho}(\xi,c_{g}r))$, where $c_{g}$ is a constant depending only on $g$. Using these estimates, we have
    \begin{align*}
    \liminf_{r \to 0} \frac{\log \nu_{g\omega}(B_{\rho}(g\xi, r))}{\log r} 
    &\ge \liminf_{r \to 0} \frac{\log (g_{*}\nu_{\omega})(gB_{\rho}(\xi,c_{g}r)) - \log \mathbb{P}_{\omega}(\cup_{n \in \Z_{+}} \{ x_{n} = g^{-1} \}) }{\log r} \\
    &= \liminf_{r \to 0} \frac{\log \nu_{\omega}(B_{\rho}(\xi,c_{g}r))}{\log r} = \liminf_{r \to 0} \frac{\log \nu_{\omega}(B_{\rho}(\xi, r))}{\log r} \\
    &\ge \frac{h}{l} -\epsilon,
    \end{align*}
    and hence $(g\omega, g\xi) \in D_{\epsilon}$, as required.
\end{proof}

\section{Invariant Percolations}

In this section, we consider invariant percolations. The arguments in this section are almost parallel to the case of random conductance models. However, we need some modifications since $\Gamma$ does not act on the space of environments $\Omega_{1}$ in this case. Throughout this section, $\Gamma$ denotes a nonelementary hyperbolic group and $G$ denotes its Cayley graph with respect to some generating set $S$ of $\Gamma$, and $E$ denotes the edge set of $G$.

\subsection{Ergodic theory of invariant percolations}

\begin{definition}[$\Gamma$-invariant percolations]
    A $\Gamma$-\textit{invariant percolation} on $G$ is a $\Gamma$-invariant probability measure on $\{0, 1\}^{E}$. Note that $\Gamma$ acts on $\{0, 1\}^{E}$ by translations. We say that a $\Gamma$-invariant percolation $\mu$ is \textit{ergodic} if it is ergodic under the $\Gamma$-action. We identify an element $\omega \in \{0, 1\}^{E}$ with a spanning subgraph of $G$ such that $e \in \omega$ if and only if $\omega(e) = 1$, and denote by $C_{\omega}(1)$ the connected component (often called cluster) of $\omega$ containing $1$. We define $\Omega_{1}$ by
    \[
    \Omega_{1} := \{ \omega \in \{0, 1\}^{E} \colon C_{\omega}(1) \, \, \text{is infinite} \}.
    \]
    Note that $\Omega_{1}$ is not $\Gamma$-invariant. In the rest of this section, we always assume that $\mu(\Omega_{1})>0$. Note that by the $\Gamma$-invariance of $\mu$, if $\omega$ has an infinite cluster with $\mu$-positive probability, then $\mu(\Omega_{1})>0$.  
    Each $\omega \in \Omega_{1}$ gives a Markov kernel $P_{\omega}$ defined on $C_{\omega}(1)$ as follows:
    \[
    p_{\omega}(g, gs) = \frac{1}{\deg_{C_{\omega}(1)}(g)} 
    \]
    if $(g,gs)\in E(C_{\omega}(1))$, and all other transition probabilities are zero. Since $\omega \in \Omega_{1}$, each $g \in C_{\omega}(1)$ has at least one neighbor in $C_{\omega}(1)$, and hence the denominator is not zero. Let $X_{\omega}$ denote the Markov chain starting from $1$ determined by $P_{\omega}$.  
\end{definition}

Let us introduce cluster relations and indistinguishability, following \cite{MR2221157} and \cite{MR2534099}. They can be used as substitutes for the $\Gamma$-action on the space of environments and its ergodicity, respectively. 

\begin{definition}[Cluster relations and indistinguishability]
    Let $\mu$ be an ergodic $\Gamma$-invariant percolation on $G$ satisfying $\mu(\Omega_{1})>0$. We define the \textit{cluster relation} on $\{0,1\}^{E}$ by
    \[
    \mathcal{R}^{cl} = \{(\omega, g\omega) \in \{0,1\}^{E} \times \{0,1\}^{E} \colon g^{-1} \in C_{\omega}(1) \}.
    \]
    This is actually an equivalence relation on $\{0,1\}^{E}$. We say that $\mu$ has \textit{indistinguishable infinite clusters} if the restricted relation $\mathcal{R}^{cl}|_{\Omega_{1}}$ is ergodic with respect to $\mu$. Recall that a countable Borel equivalence relation $\mathcal{R}$ on a standard probability space $(X, \nu)$ is called \textit{ergodic} if every Borel subset $A \subset X$ satisfying the equation
    \[
    A = \{ x\in X \colon \exists x' \in A, (x, x') \in \mathcal{R} \}
    \]
    is $\nu$-null or $\nu$-conull.
    Note that many invariant percolation models (including supercritical Bernoulli percolations) are known to have indistinguishable infinite clusters \cite{MR1742889}. 
\end{definition}

\begin{remark}
    We have defined the indistinguishability using cluster relations. It is straightforward to check that this definition is equivalent to the original definition in Section 3 of \cite{MR1742889} under the ergodicity of $\mu$. See Proposition 5 in \cite{MR2534099} for details.
\end{remark}

\begin{definition}[Path bundles]
   Let $\mu$ be a $\Gamma$-invariant percolation on $G$. We define the weighted version $\mu'$ of $\mu$ on $\Omega_{1}$ by 
   \[
   d\mu'(\omega)= \frac{\deg_{C_{\omega}(1)}(1)}{D} d\mu(\omega),
   \]
   where $D=\int_{\Omega_{1}}\deg_{C_{\omega}(1)}(1)\,d\mu(\omega)$.
   The \textit{path bundle} $\Pi$ associated with $\mu$ is $\Omega_{1} \times \Gamma^{\Z_{+}}$ as a Borel space, and we endow $\Pi$ with the probability measure $\lambda$ defined by 
   \[
   \lambda = \int_{\Omega_{1}} \delta_{\omega} \otimes \mathbb{P}_{\omega} d\mu'(\omega),
   \]
   where $\mathbb{P}_{\omega}$ denotes the distribution of sample paths of $X_{\omega}$, which is a probability measure on $\Gamma^{\Z_{+}}$. 
   Further, the \textit{shift map} $T$ on $\Pi$ is defined by $(\omega, (x_{n})_{n \in \Z_{+}}) \mapsto (x_{1}^{-1}\omega, (x_{1}^{-1}x_{n+1})_{n \in \Z_{+}})$.
\end{definition}

The following theorem is crucial for our study. 

\begin{theorem}\label{fundthmPerc}
    Let $\mu$ be an ergodic $\Gamma$-invariant percolation on $G$ with indistinguishable infinite clusters. Then the system $(\Pi, \lambda, T)$ is ergodic and measure-preserving.
\end{theorem}

Note that this type of result is well known in the context of random walks in random environments (see \cite{MR2278453}, for example). Since the proof is similar for all three cases (random conductance models, percolations, and Voronoi tessellations), we refer to Proposition \ref{pmperg} (in the case of point processes). 

This theorem implies that the speed and the entropy of the simple random walks on percolation clusters exist and they are constant $\mu$-almost surely. Note that they may be zero in some invariant percolations like wired uniform spanning forests (see Remark 4.5 in \cite{MR1802426}). However, many interesting examples are known to have positive entropy and speed as indicated in the following result due to Benjamini, Lyons, and Schramm.  

\begin{theorem}[Theorem 4.4 and Lemma 4.6 in \cite{MR1802426}]\label{BLS}
    Let $G$ be a Cayley graph of a nonamenable group and $\mu$ be an ergodic $\Gamma$-invariant percolation on $G$ with indistinguishable infinite clusters. Assume one of the following conditions:
    \begin{itemize}
        \item $\mu$ is a supercritical Bernoulli percolation.
        \item $\mu$-almost every $\omega$ has a unique infinite cluster.
        \item $\mu$-almost every $\omega$ has a cluster with at least three ends.
        \item The expected degree at $1$, conditioned on $\Omega_1$, is larger than $d-i(G)$, where $d$ is the degree of $G$ and $i(G)$ is its isoperimetric constant. 
    \end{itemize}
    Then for $\mu$-almost every $\omega \in \Omega_{1}$, the simple random walk $X_{\omega}$ on $C_{\omega}(1)$ is regular and its entropy and speed do not depend on $\omega$ and are positive.
\end{theorem}

In such cases, we can consider the corresponding harmonic measures by Theorem \ref{geodesictracking} and define the ray bundle.

\begin{definition}[Ray bundles]\label{def57}
    Let $\mu$ be an ergodic $\Gamma$-invariant percolation on $G$ with indistinguishable infinite clusters. Assume that the corresponding simple random walk has positive speed, i.e., the simple random walk $X_{\omega}$ on $C_{\omega}(1)$ has positive speed for $\mu$-almost every $\omega \in \Omega_{1}$. Note that the simple random walk $X_{\omega}$ determines the corresponding harmonic measure $\nu_{\omega}$ for $\mu$-almost every $\omega \in \Omega_{1}$ by Theorem \ref{geodesictracking}. The \textit{ray bundle} $\Xi$ associated with $\mu$ is $\Omega_{1} \times \partial G$ as a Borel space, and we endow $\Xi$ with the probability measure $\eta$ defined by 
    \[
    \eta = \int_{\Omega_{1}} \delta_{\omega} \otimes \nu_{\omega} d\mu'(\omega).
    \]
\end{definition}

Next, we define an equivalence relation which substitutes for the diagonal action of $\Gamma$ on the ray bundle and show its ergodicity with respect to $\eta$.

\begin{lemma}\label{relationerg}
     Let $\mu$ be an ergodic $\Gamma$-invariant percolation on $G$ with indistinguishable infinite clusters. Assume that the corresponding simple random walk has positive speed. Then, the following equivalence relation $\mathcal{R}_{1}$ on $\Xi$ is ergodic with respect to $\eta$:
    \[
    \mathcal{R}_{1} = \{((\omega, \xi), (g\omega, g\xi)) \in \Xi^{2} \colon g^{-1}\in C_{\omega}(1) \}.
    \]
\end{lemma}

\begin{proof}
    Let $A$ be an $\mathcal{R}_{1}$-invariant subset of $\Xi = \Omega_{1} \times \partial G$. Then the inverse image of $A$ under $\id \times \bnd$ is invariant under the shift map $T$. Since $T$ is ergodic with respect to $\lambda$ (Theorem \ref{fundthmPerc}) and $(\id \times \bnd)_{*}\lambda = \eta$, it implies that $A$ is null or conull.
\end{proof}

\subsection{The dimension formula for invariant percolations}

In this subsection, we prove the main theorem of this paper.

\begin{theorem}\label{mainPerc}
    Let $\mu$ be an ergodic $\Gamma$-invariant percolation with indistinguishable infinite clusters. Assume that the corresponding simple random walk has positive speed. Then, for $\mu$-almost every $\omega \in \Omega_{1}$, letting $\nu_{\omega}$ be the harmonic measure on $\partial G$ determined by $X_{\omega}$, we have
    \[
    \dim \nu_{\omega} = \lim_{r \to 0} \frac{\log \nu_{\omega}(B_{\rho}(\xi, r))}{\log r} = \frac{h}{l}
    \]
    for $\nu_{\omega}$-almost every $\xi \in \partial G$, where $h$ and $l$ denote the entropy and the speed of $X_{\omega}$, respectively. In particular, $\dim \nu_{\omega}$ is positive and constant for $\mu$-almost every $\omega \in \Omega_{1}$.
\end{theorem}

\begin{corollary}
    Let $\mu$ be an ergodic $\Gamma$-invariant percolation with indistinguishable infinite clusters. Assume one of the following conditions:
    \begin{itemize}
        \item $\mu$ is a supercritical Bernoulli percolation.
        \item $\mu$-almost every $\omega$ has a unique infinite cluster.
        \item $\mu$-almost every $\omega$ has a cluster with at least three ends.
        \item The expected degree at $1$, conditioned on $\Omega_1$, is larger than $d-i(G)$, where $d$ is the degree of $G$ and $i(G)$ is its isoperimetric constant. 
    \end{itemize}
    Then, for $\mu$-almost every $\omega \in \Omega_{1}$, letting $\nu_{\omega}$ be the harmonic measure on $\partial G$ determined by $X_{\omega}$, we have
    \[
    \dim \nu_{\omega} = \lim_{r \to 0} \frac{\log \nu_{\omega}(B_{\rho}(\xi, r))}{\log r} = \frac{h}{l}
    \]
    for $\nu_{\omega}$-almost every $\xi \in \partial G$. In particular, $\dim \nu_{\omega}$ is positive and constant for $\mu$-almost every $\omega \in \Omega_{1}$. 
\end{corollary}

\begin{proof}
    This follows from Theorem \ref{BLS} and Theorem \ref{mainPerc}.
\end{proof}

In the rest of this section, we prove Theorem \ref{mainPerc}. First, we have the following results as in the case of random conductance models.

\begin{theorem}\label{upperboundPerc}
    Let $\mu$ be an ergodic $\Gamma$-invariant percolation on $G$ with indistinguishable infinite clusters. Assume that the corresponding simple random walk has positive speed. Then, for $\mu$-almost every $\omega \in \Omega_{1}$ and for $\nu_{\omega}$-almost every $\xi \in \partial G$, we have
    \[
    \limsup_{r \to 0} \frac{\log \nu_{\omega}(B_{\rho}(\xi, r))}{\log r} \le \frac{h}{l}
    \]
    where $h$ and $l$ are the entropy and the speed of $X_{\omega}$, respectively. In particular, we have $\dim \nu_{\omega} \le \frac{h}{l}$ for $\mu$-almost every $\omega \in \Omega_{1}$.
\end{theorem}

\begin{proof}
    This follows from Theorem \ref{genupperbound} and the assumption that the speed is positive. 
\end{proof}

\begin{theorem}\label{lowerboundPerc}
     Let $\mu$ be an ergodic $\Gamma$-invariant percolation on $G$ with indistinguishable infinite clusters. Assume that the simple random walk on $C_{\omega}(1)$ has positive speed. Then, for every $\epsilon > 0$, the subset $D_{\epsilon}$ of $\Xi$ defined by
    \[
    D_{\epsilon} := \biggl\{(\omega, \xi) \in \Omega_{1} \times \partial G \colon \liminf_{r \to 0} \frac{\log \nu_{\omega}(B_{\rho}(\xi, r))}{\log r} \ge \frac{h}{l} - \epsilon \biggr\}  
    \]
    is $\eta$-positive, where $\eta = \int_{\Omega_{1}} \delta_{\omega} \otimes \nu_{\omega} d\mu'(\omega)$ is the probability measure defined in Definition \ref{def57}.  
\end{theorem}

\begin{proof}
    This follows from Theorem \ref{genlowerbound} and the assumption that the speed is positive.
\end{proof}

In the proof of Theorem 1.1 in \cite{MR3893268}, the ergodicity of the $\Gamma$-action on $\partial G$ is used to prove that $D_{\epsilon}$ is actually full measure. We can modify the argument by using ergodicity of the equivalence relation $\mathcal{R}_{1}$ defined in Lemma \ref{relationerg}.

\begin{proof}[Proof of Theorem \ref{mainPerc}]
   Let us show that the set $D_{\epsilon}$ in Theorem \ref{lowerboundPerc} is $\eta$-conull for every $\epsilon > 0$. It is enough to prove that $D_{\epsilon}$ is $\mathcal{R}_{1}$-invariant since $\mathcal{R}_{1}$ is ergodic with respect to $\eta$ (Lemma \ref{relationerg}) and $D_{\epsilon}$ is $\eta$-positive (Theorem \ref{lowerboundPerc}). Let $(\omega, \xi) \in D_{\epsilon}$ and let $g\in\Gamma$ satisfy $g^{-1}\in C_\omega(1)$. Let $\mathbb{P}_{\omega, g^{-1}}$ denote the distribution of sample paths starting from $g^{-1}$ following $P_{\omega}$. Then, for every Borel subset $A \subset \partial G$, we have 
    \begin{align*}
    (g_{*}\nu_{\omega})(A) 
    &= \mathbb{P}_{\omega}(\{ x \in \Gamma^{\Z_{+}} \colon \xi_{x} \in g^{-1}A\}) \ge \mathbb{P}_{\omega}(\bigcup_{n \in \Z_{+}} \{x_{n} = g^{-1} \} \cap \{ x \in \Gamma^{\Z_{+}} \colon \xi_{x} \in g^{-1}A\}) \\
    &= \mathbb{P}_{\omega}(\bigcup_{n \in \Z_{+}} \{x_{n} = g^{-1} \}) \mathbb{P}_{\omega, g^{-1}}(\{ x \in \Gamma^{\Z_{+}} \colon \xi_{x} \in g^{-1}A\}) \\
    &=  \mathbb{P}_{\omega}(\bigcup_{n \in \Z_{+}} \{x_{n} = g^{-1} \}) \mathbb{P}_{g\omega}(\{ x \in \Gamma^{\Z_{+}} \colon \xi_{x} \in A \})  
    = \mathbb{P}_{\omega}(\bigcup_{n \in \Z_{+}} \{x_{n} = g^{-1} \}) \nu_{g\omega}(A),
    \end{align*}
    where we have used the strong Markov property of $X_{\omega}$ in the second equality. The third equality follows from the equality $\mathbb{P}_{g\omega} = g_{*}\mathbb{P}_{\omega, g^{-1}}$. Note that $\mathbb{P}_{\omega}(\bigcup_{n \in \Z_{+}} \{x_{n} = g^{-1} \}) > 0$
    since $g^{-1}\in C_{\omega}(1)$.
    Further, by Lemma \ref{gammaball}, we have $\nu_{g\omega}(B_{\rho}(g\xi,r)) 
    \le \nu_{g\omega}(gB_{\rho}(\xi,c_{g}r))$ where $c_{g}$ is a constant depending only on $g$. 
    Then, we have
    \begin{align*}
    \liminf_{r \to 0} \frac{\log \nu_{g\omega}(B_{\rho}(g\xi, r))}{\log r} 
    &\ge \liminf_{r \to 0} \frac{\log (g_{*}\nu_{\omega})(gB_{\rho}(\xi,c_{g}r)) - \log \mathbb{P}_{\omega}(\bigcup_{n \in \Z_{+}} \{ x_{n} = g^{-1} \}) }{\log r} \\
    &= \liminf_{r \to 0} \frac{\log \nu_{\omega}(B_{\rho}(\xi,c_{g}r))}{\log r} = \liminf_{r \to 0} \frac{\log \nu_{\omega}(B_{\rho}(\xi, r))}{\log r} \\
    &\ge \frac{h}{l} -\epsilon
    \end{align*}
    from the above estimates and $(\omega, \xi) \in D_{\epsilon}$. Hence, it follows that $(g\omega, g\xi) \in D_{\epsilon}$, as required.
\end{proof}

\section{Invariant point processes}

\subsection{Ergodic theory of invariant point processes}
Let $G$ be a noncompact, second-countable, and unimodular locally compact group, and $K$ be its compact subgroup. Let $X$ denote the symmetric space $G / K$ and $o$ denote the base point of $X$ represented by the identity element of $G$. We fix a proper geodesic $G$-invariant metric $d_{X}$ on $X$. 

First, we introduce the notion of $G$-invariant point processes on $X$.

\begin{definition}[Invariant point processes]
    The \textit{configuration space} $M(X)$ of $X$ is defined by
    \[
    M(X) = \{ \omega \subset X \colon \omega \cap B \ \text{is finite for every compact} \ B\subset X \}.
    \]
    We always consider the weakest Borel structure on $M(X)$ such that all maps in the form of $\# \omega \cap U$ for some Borel subset $U \subset X$ are measurable.
    A $G$-\textit{invariant point process} on $X$ is a Borel probability measure on $M(X)$ which is invariant under the natural action $G \curvearrowright M(X)$.
\end{definition}

\begin{example}[Poisson point process]
    Let $m_{X}$ be a nonzero $G$-invariant Borel measure on $X$. The Poisson point process on $X$ with respect to $m_{X}$ is the $G$-invariant point process $\mu$ on $X$ characterized by the following two properties:
    \begin{itemize}
        \item For every Borel subset $U \subset X$ satisfying $m_{X}(U) \neq \infty$, the random variable defined by $\# \omega \cap U$ follows the Poisson distribution with intensity $m_{X}(U)$.
        \item For every pair of disjoint Borel subsets $U, V \subset X$, the random variables defined by $\# \omega \cap U$, $\# \omega \cap V$ are independent. 
    \end{itemize}
\end{example}

\begin{remark}
    In the study of point processes, it is natural to consider Poisson point processes with parameter $\lambda \in \mathbb{R}_{+}$, called intensity. However, in our situation, changing intensities is equivalent to changing $G$-invariant measures on $X$. Therefore, we consider only the case when $\lambda = 1$ for the sake of simplicity.
\end{remark}

The following ergodicity of Poisson point processes is well-known and crucially used in the proof of the main theorem. For the proof, we refer to Proposition 2.7 and Theorem 7.9 in \cite{MR4555892}.

\begin{proposition}[Poisson point process is ergodic]\label{PPPerg}
    Let $\mu$ be a Poisson point process on $X$ with respect to an invariant measure $m_{X}$. Then, the action $G \curvearrowright (M(X), \mu)$ is ergodic.
\end{proposition}

We can construct a random graph on $X$ from a point process as follows.

\begin{definition}[Voronoi tessellations and Delaunay graphs]
   For $\omega \in M(X)$, the \textit{Voronoi tessellation} associated to $\omega$ is the family $\mathcal{C}_{\omega}$ of subsets of $X$ defined by
   \[
   \mathcal{C}_{\omega} = (C_{x} := \{y \in X \colon d_{X}(x,y) \le d_{X}(x', y) \ \text{for every} \ x' \in \omega \})_{x \in \omega}.
   \]
   The \textit{Delaunay graph} $\mathcal{G}_{\omega}$ associated to $\omega$ is the dual graph of the Voronoi tessellation, i.e.,
   \begin{itemize}
       \item $V(\mathcal{G}_{\omega}) = \omega$
       \item $E(\mathcal{G}_{\omega}) = \{ (x, y) \in \omega^{2} \colon x\neq y,\ C_{x} \cap C_{y} \neq \varnothing \}$.
   \end{itemize}
   For a $G$-invariant point process $\mu$ on $X$, we can consider the random Delaunay graph $\mathcal{G}_{\omega}$, where the distribution of $\omega$ is given by $\mu$. In particular, when $\mu$ is Poisson point process, this $\mathcal{G}_{\omega}$ is called \textit{Poisson-Delaunay graph}.
\end{definition}

Then, the random walks on Voronoi tessellations are naturally formalized as the simple random walks on Delaunay graphs.

\subsection{Palm measures}
In this subsection, we introduce the notion of the Palm measures. It gives a general way to condition point processes so that they contain the base point $o$. Let $K\backslash M_{o}(X)$ be the quotient space of $M_{o}(X) = \{ \omega \in M(X) \colon o \in \omega \}$ by the left action by the stabilizer $K$ of $o$. Note that this is a standard Borel space since the action of the compact group $K$ is smooth (see Proposition 7.10 in \cite{MR4555892} for details).
\begin{definition}\label{palmmeasure}
    Let $\mu$ be a $G$-invariant point process on $X$. Fix a relatively compact neighborhood $U\subset X$ of the base point with positive volume. We assume that $0<\mathbb{E}_{\mu}[\#(\omega\cap U)]<\infty$, and refer to this as finite positive intensity. We define \textit{the Palm measure} $\mu_{o}$ of $\mu$ on $K \backslash M_{o}(X)$ by
    \[
    \mu_{o}(A) = \frac{\mathbb{E}_{\mu}[\#\{ gK \in \omega \cap U \colon g^{-1}\omega \in A \}]}{\mathbb{E}_{\mu}[\#\{gK \in \omega \cap U\}]}
    \]
    for each $K$-invariant Borel subset $A \subset M_{o}(X)$. Note that the condition in the right hand side is well-defined since $A$ is $K$-invariant. 
\end{definition}

In the rest of this section, we consider point processes of finite positive intensity. In the case of Poisson point processes, the Palm measure is obtained by adding the origin to a sample of the point process and taking the quotient by $K$. For the proof, we refer to Section 9.2 in \cite{MR3791470}.

\begin{proposition}\label{palmppp}
    Let $\mu$ be a Poisson point process on $X$. Then, its Palm measure $\mu_{o}$ is the distribution of $K(\omega\cup\{o\})$ when $\omega$ has distribution $\mu$.
\end{proposition}

For $\omega \in M_{o}(X)$ such that $\mathcal{G}_\omega$ is locally finite, let $\mathbb{P}_{\omega}$ denote the Borel probability measure on $X^{\Z_{+}}$ obtained as the distribution of sample paths of the simple random walk on $\mathcal{G}_{\omega}$ starting at $o$. Since $\mathbb{P}_{k\omega}=k_{*}\mathbb{P}_{\omega}$, the pushforward of $\delta_{\omega}\otimes\mathbb{P}_{\omega}$ to $K\backslash(M_o(X)\times X^{\Z_+})$ depends only on $K\omega$. We use this convention for the measures on the path bundle below. 

\subsection{Rerooting relations}
In this section, we develop the ergodic-theoretic machinery for proving the main theorem. Our argument is based on the paper by Ab\'ert-Mellick \cite{MR4555892}, which concerns a problem in orbit equivalence theory. Note that this part does not depend on the hyperbolicity of the space and works for arbitrary symmetric spaces. 

First, we define rerooting relations, which can be thought as analogues of cluster relations in the case of invariant percolations on discrete groups.
\begin{definition}\label{rerootingdef}
    The \textit{rerooting relation} $\mathcal{R}_{o}$ is an equivalence relation on $K \backslash M_{o}(X)$ defined by 
    \[
    \mathcal{R}_{o} = \{ (\omega, g^{-1}\omega) \in (K \backslash M_{o}(X))^{2} \colon gK \in \omega \}.
    \]
    Note that the condition in the right hand side is well-defined by the $K$-invariance.
\end{definition}

\begin{proposition}\label{rerootingerg}
    In the above setting, the rerooting relation $\mathcal{R}_{o}$ is nonsingular with respect to the Palm measure $\mu_{o}$.
    Further, if the action $G \curvearrowright (M(X), \mu)$ is ergodic, then the rerooting relation $\mathcal{R}_{o}$ is ergodic with respect to the Palm measure $\mu_{o}$. In particular, the rerooting relation is ergodic with respect to the Palm measure of Poisson point processes. 
\end{proposition}

Here, nonsingularity means that the $\mathcal{R}_{o}$-saturation of every $\mu_{o}$-null set is $\mu_{o}$-null. Before the proof, let us recall that a countable Borel equivalence relation $\mathcal{R}$ on a standard probability space $(X, \nu)$ is called \textit{ergodic} if every Borel subset $A \subset X$ satisfying the equation
    \[
    A = \{ x\in X \colon \exists x' \in A, (x, x') \in \mathcal{R} \}
    \]
    is $\nu$-null or $\nu$-conull.

\begin{proof}
    For the first part, let $A \subset K\backslash M_{o}(X)$ be a $\mu_{o}$-null Borel subset. For each $R>0$, the Palm mass-transport identity (Section 7.2 in \cite{MR4555892}) gives
    \[
    \int \sum_{gK\in\omega\cap B_X(o,R)} \mathbf{1}_A(g^{-1}\omega)\,d\mu_o(\omega)
    =\int \mathbf{1}_A(\omega)\#(\omega\cap B_X(o,R))\,d\mu_o(\omega)=0.
    \]
    Here the sum is over the points of $\omega$, including those giving the same element of $K\backslash M_o(X)$. Taking $R\in\N$, we conclude that the $\mathcal{R}_o$-saturation of $A$ is $\mu_o$-null.

    For the second part, let $A \subset K\backslash M_{o}(X)$ be a $\mu_{o}$-positive and $\mathcal{R}_{o}$-invariant Borel subset. We define a subset $A'$ of $M(X)$ by
    \[
    A' = \{ \omega \in M(X) \colon \exists gK \in \omega, g^{-1}\omega \in A \},
    \]
    which is $\mu$-positive by the definition of the Palm measure $\mu_{o}$. We want to show that
    \[
    A' = \{ \omega \in M(X)\setminus\{\varnothing\} \colon \forall gK \in \omega, g^{-1}\omega \in A \}.
    \]
    Indeed, once this equality is established, it follows that $A'$ is $G$-invariant. Then, by the ergodicity of the action, $A'$ is actually $\mu$-conull, and it implies the original subset $A$ is $\mu_{o}$-conull, again by the definition of $\mu_{o}$. Note that the right hand sides make sense by the $K$-invariance.

    In the rest of the proof, we prove the above equality. It is obvious that the right hand side is contained in $A'$. Let $\omega$ be an element of $A'$. Then, there exists $gK \in \omega$ satisfying $g^{-1}\omega \in A$. For every $hK \in \omega$, we have $(g^{-1}\omega, h^{-1}\omega) \in \mathcal{R}_{o}$ since
    \[
    h^{-1}\omega = h^{-1}gg^{-1}\omega = (g^{-1}h)^{-1}g^{-1}\omega
    \]
    and $g^{-1}hK \in g^{-1}\omega$. Since $A$ is $\mathcal{R}_{o}$-invariant, it follows that $h^{-1}\omega$ also belongs to $A$. This completes the proof.
\end{proof}

If the action $G\curvearrowright(M(X),\mu)$ is essentially free, i.e., $\mu$-almost every configuration has trivial stabilizer in $G$, then the rerooting relation $\mathcal{R}_o$ preserves the Palm measure $\mu_o$. In this case, for $\mu_o$-almost every $\omega$, the points of $\omega$ are in bijection with its rerooting class, so the assertion follows from the Palm mass-transport identity. This applies, for example, to the Poisson point process on $\mathbb{H}^d$ with $G=\operatorname{Isom}(\mathbb{H}^d)$.

\subsection{Path bundles}

Next, we define the path bundle over the configuration space $K \backslash M_{o}(X)$, which records the information of the simple random walk on the Delaunay graphs $\mathcal{G}_{\omega}$.

\begin{definition}\label{pathbundle}
    Let $\mu$ be a $G$-invariant point process on $X = G/K$ satisfying $\mathbb{E}_{\mu_{o}}[\deg_{\mathcal{G}_{\omega}}(o)] < \infty$. By the nonsingularity in Proposition \ref{rerootingerg}, $\mathcal{G}_\omega$ is locally finite for $\mu_o$-almost every $\omega$. The \textit{path bundle} $\Pi$ over the configuration space $K \backslash M_{o}(X)$ is the Borel space $K \backslash (M_{o}(X) \times X^{\Z_{+}})$ defined as the quotient of $M_{o}(X) \times X^{\Z_{+}}$ by the diagonal $K$-action, equipped with the probability measure $\lambda$ defined by 
    \[
    \lambda = \frac{1}{\mathbb{E}_{\mu_{o}}[\deg_{\mathcal{G}_{\omega}}(o)]} \int_{K\backslash M_{o}(X)} \deg_{\mathcal{G}_{\omega}}(o) \delta_{\omega} \otimes \mathbb{P}_{\omega} \ d\mu_{o}(\omega),
    \]
    where $\delta_\omega\otimes\mathbb{P}_\omega$ is interpreted on the diagonal quotient as above. Note that $\mathbb{E}_{\mu_{o}}[\deg_{\mathcal{G}_{\omega}}(o)]$ is finite for Poisson-Delaunay graphs on Riemannian symmetric spaces (see Theorem 3.3 in \cite{MR3880013}). 

Further, we can define the \textit{shift map} $T$ on $\Pi$ by 
    \[
    T((\omega, (g_{n}K)_{n \in \Z_{+}})) = (g_{1}^{-1}\omega, (g_{1}^{-1}g_{n+1}K)_{n \in \Z_{+}}).
    \]
    This map is defined on the full-measure set of paths starting at $o$ and following edges of $\mathcal{G}_\omega$. It does not depend on the chosen coset representatives since we take the quotient by the left $K$-action.
\end{definition} 

\begin{proposition}\label{pmperg}
     Let $\mu$ be a $G$-invariant point process on $X$ with finite expected degree, i.e., $\mathbb{E}_{\mu_{o}}[\deg_{\mathcal{G}_{\omega}}(o)] < \infty$. Then, the system $(\Pi, \lambda, T)$ is measure-preserving. Further, if $\mu$ is ergodic, the system is also ergodic.
\end{proposition}

This type of result is established in the theory of unimodular random networks by Aldous and Lyons (Theorems 4.1 and 4.6 in \cite{AL07}). For embedded Delaunay graphs in Riemannian symmetric spaces, see also Theorems 1.4 and 1.9 in \cite{MR3880013}. We give the proof in the present setting for completeness.

\begin{proof}
    The first part follows from the Palm mass-transport identity applied to neighboring vertices, since the degree in the definition of $\lambda$ cancels the reciprocal degree in the transition probability. Before the proof of the ergodicity, we introduce the following notation:
    \begin{itemize}
        \item $\mu'_{o}$ is the probability measure on $K \backslash M_{o}(X)$ defined by $d\mu'_{o} = \frac{\deg_{\mathcal{G}_{\omega}}(o)}{\mathbb{E}_{\mu_{o}}[\deg_{\mathcal{G}_{\omega}}(o)]} d\mu_{o}$.
        \item $\Pi_{\Z} = K \backslash (M_{o}(X) \times X^{\Z})$.
        \item $\lambda_{\Z} = \int_{K \backslash M_{o}(X)} \delta_{\omega} \otimes \mathbb{P}^{\Z}_{\omega} d\mu'_{o}(\omega)$.
        \item $U(\omega, (g_{n}K)_{n \in \Z}) = (g_{1}^{-1}\omega, (g_{1}^{-1}g_{n+1}K)_{n \in \Z})$.
    \end{itemize}
    Here $\mathbb{P}^{\Z}_{\omega}$ is the law of two independent simple random walks from $o$, one in each time direction. Then, the system $(\Pi_{\Z}, \lambda_{\Z}, U)$ is also measure-preserving and the original system is a factor of it. Therefore, it is enough to show that the bilateral version $(\Pi_{\Z}, \lambda_{\Z}, U)$ is ergodic.
    For each Borel subset $B \subset \Pi_{\Z}$, we define a function $f_{B}$ on $K \backslash M_{o}(X)$ by
    \[
    f_{B}(\omega) = \mathbb{P}^{\Z}_{\omega}(\{x \in X^{\Z} \colon (\omega, x) \in B\}).
    \]
    Let $A$ be a $U$-invariant subset of $\Pi_{\Z}$. We first show that $f_{A}(\omega) \in \{0, 1\}$ for $\mu'_{o}$-almost every $\omega \in K \backslash M_{o}(X)$. For $\epsilon > 0$, there exist $N \in \N$ and a Borel subset $A_{N}$ of $\Pi_{\Z}$ such that 
    \begin{itemize}
        \item $\lVert f_{A} - f_{A_{N}} \rVert_{1} \le \lambda_{\Z}(A \, \triangle \, A_{N}) < \epsilon$, and
        \item $A_{N}$ is an event depending only on $(\omega, (x_{n})_{-N \le n \le N})$,
    \end{itemize}
    where $\lVert \cdot \rVert_{1}$ denotes the norm on $L^{1}_{\mu'_{o}}(K\backslash M_o(X))$. Since $A$ is $U$-invariant, we have 
    \begin{align*}
        \lVert f_{A} - f_{A}^{2} \rVert_{1} 
        &= \lVert f_{A} - f_{U^{N}A} f_{U^{-N}A} \rVert_{1} \le \lVert f_{A} - f_{U^{N}A_{N}} f_{U^{-N}A_{N}} \rVert_{1} + 2\epsilon \\
        &\le \lVert f_{A} - f_{U^{N}A_{N} \cap U^{-N}A_{N}} \rVert_{1} + 2\epsilon \\
        &\le \lVert f_{A} - f_{U^{N}A \cap U^{-N}A}\rVert_{1} + 4\epsilon = \lVert f_{A} - f_{A}\rVert_{1} + 4\epsilon = 4\epsilon.
    \end{align*}
    Note that the second inequality follows from the fact that the events $U^{N}A_{N}$ and $U^{-N}A_{N}$ are independent, conditionally on $\omega$. Since $\epsilon$ can be taken arbitrarily, the above estimate implies that $f_{A}(\omega) \in \{0, 1\}$ for $\mu'_{o}$-almost every $\omega \in K \backslash M_{o}(X)$. 
    Therefore, it is enough to show that the subset $F$ of $K \backslash M_{o}(X)$ defined by 
    \[
    F = \{ \omega \in K \backslash M_{o}(X) \colon f_{A}(\omega) = 1 \}
    \]
    is $\mathcal{R}_{o}$-invariant modulo null sets, since $\mathcal{R}_{o}$ is ergodic with respect to $\mu'_{o}$. As $f_A$ takes values in $\{0,1\}$, we have $\mathbf{1}_A(\omega,x)=f_A(\omega)$ for $\lambda_{\Z}$-almost every $(\omega,x)$. The $U$-invariance of $A$ and measure preservation therefore give
    \[
    f_A(\omega)=\frac{1}{\deg_{\mathcal{G}_\omega}(o)}\sum_{\substack{gK\in\omega\\(o,gK)\in E(\mathcal{G}_\omega)}}f_A(g^{-1}\omega)
    \]
    for $\mu'_o$-almost every $\omega$. Thus, whenever $f_A(\omega)=1$, the same holds at every neighboring root. Since $\mathcal{R}_o$ is nonsingular with respect to $\mu_o$ (Proposition \ref{rerootingerg}) and $\mu_o$ and $\mu_o^{\prime}$ are equivalent, we may remove the saturation of the exceptional null set. Connectivity of $\mathcal{G}_\omega$ then shows that $F$ is $\mathcal{R}_o$-invariant on the resulting conull set, as required.
 \end{proof}

As a corollary, we can show that the walks are regular and the speed and the entropy are deterministic.

\begin{theorem}\label{speedentropy}
    Let $\mu$ be an ergodic $G$-invariant point process on $X$ satisfying the following:
    \begin{itemize}
        \item[(i)] $\mathbb{E}_{\mu_{o}}[\deg_{\mathcal{G}_{\omega}}(o) \log \deg_{\mathcal{G}_{\omega}}(o)] < \infty$.
        \item[(ii)] $\mathbb{E}_{\mu_{o}}[{\sum_{x \in \omega, (o,x) \in E(\mathcal{G}_{\omega})} d_{X}(o,x)}] < \infty$.
    \end{itemize}
    Then, for $\mu_{o}$-almost every $\omega$, the simple random walk on $\mathcal{G}_{\omega}$ starting at $o$ is regular, and the speed and the entropy are deterministic (i.e., $l, h$ do not depend on configurations almost surely). 
\end{theorem}

\begin{proof}
    For $n \in \Z_{+}$, we define functions on $\Pi$ as follows:
 \begin{itemize}
     \item $\phi_{n}(\omega, x) = d_{X}(o, x_{n})$,
     \item $\psi_{n}(\omega, x) = -\log \mathbb{P}_{\omega}(X_{n} = x_{n})$.
 \end{itemize}
 Note that these functions are well-defined since the right hand sides are invariant under the diagonal $K$-action on $M_{o}(X) \times X^{\Z_{+}}$.
 Then, we have 
\begin{align*}
    \phi_{m+n}(\omega, x) &= d_{X}(o, x_{m+n}) \le d_{X}(o, x_{m}) + d_{X}(x_{m}, x_{m+n}) = \phi_{m}(\omega, x) + \phi_{n}(T^{m}(\omega, x))
\end{align*}
and 
\begin{align*}
     \psi_{m+n}(\omega, x) 
     &= -\log \mathbb{P}_{\omega}(X_{m+n} = x_{m+n}) \\
     &\le -\log \mathbb{P}_{\omega}(X_{m} = x_{m}) - \log \mathbb{P}_{\omega}(X_{m+n}=x_{m+n}\mid X_m=x_m) \\
     &= \psi_{m}(\omega, x) + \psi_{n}(T^{m}(\omega, x)).
\end{align*}  
Therefore, by applying Kingman's ergodic theorem to $\phi_{n}$ and $\psi_{n}$, we obtain the existence and constancy of the speed and the entropy. Conditions (i) and (ii) guarantee the required integrability. The integrability of $d_X(x_0,x_1)$, stationarity, and the Borel--Cantelli lemma also give $d_X(x_n,x_{n+1})/n\to0$ almost surely, which completes the proof of regularity.
\end{proof}

\subsection{Ray bundles}

In this subsection, we consider the case when 
$X$ is Gromov hyperbolic and the simple random walk on the Delaunay graph $\mathcal{G}_{\omega}$ has positive speed. Under these assumptions, we can consider the harmonic measure determined by the simple random walk. Benjamini-Krauz-Paquette proved that Poisson--Voronoi tessellations on real hyperbolic spaces satisfy these assumptions as follows.

\begin{theorem}[Theorem 1.1 in \cite{MR4341072}]\label{BKP}
    Let $\mu$ be a Poisson point process on the hyperbolic space $\mathbb{H}^{d}$ with $d \ge 2$. Then, for $\mu_{o}$-almost every $\omega \in M_{o}(X)$, the simple random walk on $\mathcal{G}_{\omega}$ starting at $o$ has positive speed almost surely. 
\end{theorem}

Further, the following result of Paquette gives a sufficient condition for the positivity of the speed in a fairly general setting. A Palm version is called a \textit{distributional lattice} if both the degree and the volume of the Voronoi cell at $o$ have finite expectation (Definition 1.5 in \cite{MR3880013}). 

\begin{theorem}[Corollary 1.10 in \cite{MR3880013}]\label{PSrankone}
    Let $X = G / K$ be a rank one Riemannian symmetric space of noncompact type. Let $\mu$ be an invariant point process on $X$ whose Palm version is a distributional lattice. Then, for $\mu_{o}$-almost every $\omega$, the simple random walk on the associated Delaunay graph $\mathcal{G}_{\omega}$ has positive speed.
\end{theorem}

\begin{definition}
  Let $X = G/K$ be a symmetric space with a proper geodesic $G$-invariant Gromov hyperbolic metric $d_{X}$, and $\mu$ be an ergodic $G$-invariant point process on $X$. Assume that $\mu$ satisfies the following:
\begin{itemize}
\item[(i)] $\mathbb{E}_{\mu_{o}}[\deg_{\mathcal{G}_{\omega}}(o) \log \deg_{\mathcal{G}_{\omega}}(o)] < \infty$.
\item[(ii)] $\mathbb{E}_{\mu_{o}}[{\sum_{x \in \omega, (o,x) \in E(\mathcal{G}_{\omega})} d_{X}(o,x)}] < \infty$.
\item[(iii)] For $\mu_{o}$-almost every $\omega$, the simple random walk on $\mathcal{G}_{\omega}$ starting at $o$ has positive speed.
\end{itemize}
  Then, we define the \textit{ray bundle} $\Xi$ over the configuration space $K \backslash M_{o}(X)$ as the Borel space $K \backslash (M_{o}(X) \times \partial X)$ defined as the quotient of $M_{o}(X) \times \partial X$ by the diagonal $K$-action, equipped with the probability measure $\eta$ defined by 
   \[
   \eta = \frac{1}{\mathbb{E}_{\mu_{o}}[\deg_{\mathcal{G}_{\omega}}(o)]} \int_{K\backslash M_{o}(X)} \deg_{\mathcal{G}_{\omega}}(o) \delta_{\omega} \otimes \nu_{\omega} \ d\mu_{o}(\omega),
   \]
   where $\nu_{\omega}$ denotes the harmonic measure on $\partial X$ determined by the simple random walk on $\mathcal{G}_{\omega}$, and $\delta_\omega\otimes\nu_\omega$ is pushed forward to the diagonal quotient. Equivalently, $\eta$ is the pushforward of $\lambda$ under the boundary map $K \backslash (\id \times \bnd)$.
\end{definition}

The following ergodicity plays an important role in the proof of the main theorem.

\begin{theorem}
   We define an equivalence relation $\mathcal{R}^{\Xi}_{o}$ on $\Xi$ by 
   \[
   \mathcal{R}^{\Xi}_{o} = \{ ((\omega, \xi), (g^{-1}\omega, g^{-1}\xi)) \in \Xi^{2}  \colon gK\in\omega \}.
   \]
   Then, the relation $\mathcal{R}^{\Xi}_{o}$ is ergodic with respect to the probability measure $\eta$.
\end{theorem}

\begin{proof}
    Let $A \subset \Xi$ be an $\mathcal{R}^{\Xi}_{o}$-invariant subset with positive $\eta$-measure. Since $\eta$ is the pushforward of $\lambda$, it is enough to show that the inverse image of $A$ in $\Pi$ is $\lambda$-conull. It is straightforward to check that the inverse image is invariant under the shift map $T$ using the $\mathcal{R}^{\Xi}_{o}$-invariance of $A$. By the ergodicity of the system $(\Pi, \lambda, T)$, it follows that the inverse image of $A$ under the boundary map is $\lambda$-conull. Note that the inverse image is not null since we assume that $A$ has positive measure.
\end{proof}

\subsection{Dimension formula for Poisson-Delaunay graphs}  

In this section, we prove exact dimensionality and the dimension formula for the simple random walks on Poisson--Voronoi tessellations. The lower bound complements the upper bound for $\mathbb{H}^d$ in Section 12.1, Lemma 14 of \cite{carrasco2017furstenbergtypeformulaspeed}.

\begin{theorem}\label{mainthmhyp}
    Let $X = \mathbb{H}^{d}$ be the $d$-dimensional hyperbolic space with $d \ge 2$, and $\mu_{o}$ be the Palm measure of the Poisson point process on $X$ with respect to a $G$-invariant measure $m_{X}$. Then, for $\mu_{o}$-almost every $\omega$, the harmonic measure determined by the simple random walk on the Delaunay graph $\mathcal{G}_{\omega}$ is exact dimensional and satisfies 
    \[
    \dim \nu_{\omega} = \frac{h}{l},
    \]
    where $h$ and $l$ denote the entropy and the speed of the simple random walk on $\mathcal{G}_{\omega}$, respectively.
\end{theorem}

In fact, we can prove such results in fairly general settings.

\begin{theorem}\label{mainthmgen}
    Let $G$ be a noncompact second-countable unimodular locally compact group and $K$ be its compact subgroup such that a proper geodesic $G$-invariant metric $d_{X}$ on $X = G/K$ is Gromov hyperbolic. Let $\mu$ be an ergodic $G$-invariant point process of finite positive intensity on $X$, and $\mu_{o}$ denote the Palm measure of $\mu$. Assume the following conditions hold:
    \begin{itemize}
        \item[(i)] $\mathbb{E}_{\mu_{o}}[\deg_{\mathcal{G}_{\omega}}(o) \log\deg_{\mathcal{G}_{\omega}}(o)] < \infty$.
        \item[(ii)] $\mathbb{E}_{\mu_{o}}[{\sum_{x \in \omega, (o,x) \in E(\mathcal{G}_{\omega})} d_{X}(o,x)}] < \infty$.
        \item[(iii)] There exists $\alpha > 0$ such that 
        \[
        \sum_{n \in \Z_{+}} \mu_{o}\{ \omega \in K \backslash M_{o}(X) \colon \#(\omega \cap B_{X}(o, n)) > \alpha^{n} \} < \infty.
        \]  
        \item[(iv)] For $\mu_{o}$-almost every $\omega$, the simple random walk on $\mathcal{G}_{\omega}$ starting at $o$ has positive speed.
    \end{itemize}
    Then, for $\mu_{o}$-almost every $\omega$, the harmonic measure $\nu_{\omega}$ determined by the simple random walk on $\mathcal{G}_{\omega}$ starting at $o$ is exact dimensional and satisfies
    \[
    \dim \nu_{\omega} = \frac{h}{l},
    \]
    where $h$ and $l$ denote the entropy and the speed of the walk, respectively.
\end{theorem}

Before giving the proof of Theorem \ref{mainthmgen}, we check that the latter theorem implies the former one. 

\begin{proof}[Proof that Theorem \ref{mainthmgen} implies Theorem \ref{mainthmhyp}] As in Proposition \ref{PPPerg}, a Poisson point process $\mu$ with respect to $m_{X}$ is ergodic. It is easy to check that Poisson point processes satisfy (i) and (ii). Further, the Palm measure $\mu_{o}$ satisfies condition (iv) by Theorem \ref{BKP} and Proposition \ref{palmppp}. Therefore, it is enough to show that it satisfies the condition (iii) as well. This follows from the tail estimate for Poisson distributions and Proposition \ref{palmppp}. Indeed, by the tail estimate for Poisson distributions and Proposition \ref{palmppp}, there exist $C, C' \in \R_{+}$ such that 
\[
\mu_{o}(\{ \omega \in M_{o}(X) \colon \# \omega \cap B_{X}(o, n) > 2 V(n) \}) \le Ce^{-C'V(n)}
\]    
where $V(r)=m_X(B_X(o,r))$ is the volume of the ball for the intensity measure $m_X$. Since the volume function $V(r)$ has exponential growth, this implies the condition (iii).
\end{proof}

We also note that Theorem \ref{mainthmgen} implies Theorem \ref{mainthmrankone} by the result of Paquette (Theorem \ref{PSrankone}). Indeed, the Palm formula gives $\mathbb{E}_{\mu_o}[\operatorname{Vol}(C_o)]=\iota^{-1}$, where $\iota>0$ is the intensity relative to Riemannian volume. Thus the finite expected degree in (i) makes the Palm version a distributional lattice.

\subsection{The upper bound}
In this subsection, we give the upper bound for the dimensions of harmonic measures. The proof is similar to the case of random walks on hyperbolic groups \cite{MR2286061}. For Poisson-Delaunay walks on $\mathbb{H}^d$, the corresponding Hausdorff-dimension upper bound is already contained in \cite[Section 12.1, Lemma 14]{carrasco2017furstenbergtypeformulaspeed}. 

\begin{theorem}\label{upperbound}
    Let $\mathcal{G}$ be a locally finite graph in $X$ containing $o$ as its vertex, and let $\mathbb{P}$ denote the distribution of sample paths of the simple random walk on $\mathcal{G}$ starting at $o$. Assume that the simple random walk on $\mathcal{G}$ is regular and has positive speed with respect to $d_{X}$. Then, the harmonic measure $\nu = \bnd_{*}\mathbb{P}$ satisfies 
    \[
    \limsup_{r \to 0} \frac{\log \nu(B_{\rho}(\xi, r))}{\log r} \le \frac{h}{l}
    \]
    for $\nu$-almost every $\xi \in \partial X$, where $h$ and $l$ denote the entropy and the speed of the simple random walk on $\mathcal{G}$, respectively.
\end{theorem}

\begin{proof}
    Let $\mathbb{P}$ denote the probability on $X^{\Z_{+}}$ defined as the distribution of the simple random walk on $\mathcal{G}$ starting at $o$. For $0<\epsilon<\min\{l,1/2\}$ and $N \in \N$, we define the set 
    \[
    A_{\epsilon, N} = \{ x \in X^{\Z_{+}} \colon (x_{n+1} \,|\, x_{n}) \ge (l-\epsilon)n \, \, \text{and} \, \, \mathbb{P}(X_{n} = x_{n}) \ge e^{-(h+\epsilon)n} \, \, \text{for every} \, \, n \ge N \}.
    \]
    Then, for each $\epsilon$, there exists $N_{\epsilon}$ such that $\mathbb{P}(A_{\epsilon, N_{\epsilon}}) \ge 1-\epsilon$ by the regularity of the walk. Let $A_{\epsilon}$ denote this $A_{\epsilon, N_{\epsilon}}$. We first show that the limit
    \[
    \lim_{n \to \infty} \mathbb{P}(A_{\epsilon}\,|\,\{x_n=z_n\})
    \]
    exists for almost every $z \in A_{\epsilon}$. This follows from the Markov property of the walk. Indeed, for $n>N_{\epsilon}$, we have
    \[
    \lim_{n \to \infty} \mathbb{P}(A_{\epsilon}\,|\,\{x_n=z_n\})
    =
    \lim_{n \to \infty} \mathbb{P}(A_{\epsilon}\,|\,\{x_m=z_m \text{ for every }m\ge n\})
    =
    \mathbb{P}(A_{\epsilon}\,|\,\operatorname{tail}(z))
    \]
    by the martingale convergence theorem, where $\operatorname{tail}$ denotes the projection to the tail boundary. We define
    \[
    A'_{\epsilon}=\{z \in A_{\epsilon}: \mathbb{P}(A_{\epsilon}\,|\,\operatorname{tail}(z))>\epsilon\}.
    \]
    Then, we have
    \begin{align*}
    1-\epsilon
    &\le \mathbb{P}(A_{\epsilon})
    =\int \mathbb{P}(A_{\epsilon}\,|\,\operatorname{tail}(x))\,d\mathbb{P}(x)\\
    &=\int_{A'_{\epsilon}} \mathbb{P}(A_{\epsilon}\,|\,\operatorname{tail}(x))\,d\mathbb{P}(x)
    +\int_{(A'_{\epsilon})^c} \mathbb{P}(A_{\epsilon}\,|\,\operatorname{tail}(x))\,d\mathbb{P}(x)\\
    &\le \mathbb{P}(A'_{\epsilon})+\epsilon,
    \end{align*}
    and hence $\mathbb{P}(A'_{\epsilon})\ge1-2\epsilon$. Recall that for $z \in A_{\epsilon}$, there exists $C > 0$ such that 
    \[
    A_{\epsilon} \cap \{x_{n} = z_{n}\} \subset \{ x \in X^{\Z_{+}} \colon \xi_{x} \in B_{\rho}(\xi_{z}, Ce^{-(l-\epsilon)n})\}
    \]
    for every $n \ge N_{\epsilon}$. Then, for $z \in A'_{\epsilon}$, we have 
    \begin{align*}
        \limsup_{r \to 0} \frac{\log \nu(B_{\rho}(\xi_{z}, r))}{\log r} 
        &= \limsup_{n \to \infty} \frac{\log \nu(B_{\rho}(\xi_{z}, Ce^{-(l-\epsilon)n}))}{-(l-\epsilon)n} \\
        &\le \limsup_{n \to \infty} \frac{\log \mathbb{P}(A_{\epsilon} \cap \{x_{n} = z_{n}\})}{-(l-\epsilon)n} \\
        &= \limsup_{n \to \infty} \frac{\log \mathbb{P}(\{x_{n}=z_{n}\})}{-(l-\epsilon)n} \\
        &\le \frac{h+\epsilon}{l-\epsilon}.
    \end{align*}
    Since $\epsilon$ can be taken arbitrarily and $\mathbb{P}(A'_{\epsilon})\ge1-2\epsilon$, this implies  
    \[
    \limsup_{r \to 0} \frac{\log \nu(B_{\rho}(\xi, r))}{\log r} \le \frac{h}{l}
    \]
    for $\nu$-almost every $\xi \in \partial X$.
    We also have $\dim \nu \le \frac{h}{l}$ by Lemma \ref{Frostman}. 
\end{proof}

\subsection{The lower bound}
To establish the lower bound in this setting, we need to have an analogue of Theorem \ref{genlowerbound}. Since Delaunay graphs do not have bounded degree in general, we have to modify the argument using the growth condition.
First, we give an estimate of the number of elements of $\omega$ around trajectories, using the growth condition (iii).

\begin{lemma}\label{growthcontrol}
    Let $\epsilon, \delta$ be positive numbers. Assume that the $G$-invariant point process $\mu$ is ergodic and satisfies conditions (i)--(iii) in Theorem \ref{mainthmgen}. Then, for $\lambda$-almost every $(\omega, (z_{n})_{n \in \Z_{+}}) \in \Pi$, it holds that
    \[
    \# (\omega \cap B_{X}(z_{n}, \epsilon n + \delta)) \le \alpha^{2 \epsilon n}
    \]
    for all large enough $n \in \Z_{+}$.
\end{lemma}

\begin{proof}
    We may assume that $\alpha>1$. For $M\in\N$, let
    \[
    E_M=\{\omega\in K\backslash M_o(X):\#(\omega\cap B_X(o,m))\le\alpha^m \text{ for every integer }m\ge M\}.
    \]
    By condition (iii) and the Borel--Cantelli lemma, $\bigcup_{M\in\N}E_M$ is $\mu_o$-conull. Since $\mu_o$ and $\mu'_o$ are equivalent, we can fix $M\in\N$ such that $p:=\mu'_o(E_M)>0$. Let
    \[
    \widehat E_M=\{(\omega,(z_n)_{n\in\Z_+})\in\Pi:\omega\in E_M\},
    \]
    with $\lambda(\widehat E_M)=p$.
    By Proposition \ref{pmperg} and Birkhoff's ergodic theorem, for $\lambda$-almost every $(\omega,(z_n)_{n\in\Z_+})\in\Pi$,
    \[
    \frac{1}{n}\sum_{j=0}^{n-1}\mathbf{1}_{\widehat E_M}\bigl(T^j(\omega,(z_k)_{k\in\Z_+})\bigr)\to p.
    \]
    For such a path and all sufficiently large $n$, define
    \[
    \tau_n := \max\left\{0\le j\le n\colon T^j(\omega,(z_k)_{k\in\mathbb Z_+})\in \widehat E_M\right\},
    \]
then the above convergence and $p>0$ imply that $\tau_n/n\to1$.
    By condition (ii),
    \[
    \int_\Pi d_X(z_0,z_1)\,d\lambda
    =
    \frac{1}{\mathbb{E}_{\mu_o}[\deg_{\mathcal G_\omega}(o)]}
    \mathbb{E}_{\mu_o}\left[\sum_{\substack{x\in\omega\\(o,x)\in E(\mathcal G_\omega)}}d_X(o,x)\right]
    <\infty.
    \]
    Hence, by Birkhoff's ergodic theorem and $\tau_n/n\to1$,
    \[
    \frac{1}{n}\sum_{j=\tau_n}^{n-1}d_X(z_j,z_{j+1})\to0.
    \]
    In particular, $d_X(z_{\tau_n},z_n)=o(n)$. Therefore, for all sufficiently large $n$,
    \[
    B_X(z_n,\epsilon n+\delta)\subset B_X(z_{\tau_n},\lfloor2\epsilon n\rfloor)
    \]
    and $\lfloor2\epsilon n\rfloor\ge M$. Since the environment seen from $z_{\tau_n}$ belongs to $E_M$, the definition of $E_M$ gives
    \[
    \#(\omega\cap B_X(z_n,\epsilon n+\delta))
    \le \#(\omega\cap B_X(z_{\tau_n},\lfloor2\epsilon n\rfloor))
    \le \alpha^{2\epsilon n}.
    \]
    This completes the proof.
\end{proof}
    
For an embedded graph with the above property, we have the following estimate.

\begin{lemma}\label{lowerboundlem}
   Let $\mathcal{G}$ be a locally finite graph in $X$ containing $o$ as its vertex, and let $\mathbb{P}$ denote the distribution of sample paths of the simple random walk on $\mathcal{G}$ starting at $o$. Assume that the following properties hold:
   \begin{itemize}
       \item The simple random walk on $\mathcal{G}$ is regular and has positive speed with respect to $d_{X}$.
       \item There exists $\alpha > 0$ such that, for any positive numbers $\epsilon, \delta$ and $\mathbb{P}$-almost every $(z_{n})_{n \in \Z_{+}}$, it holds that
       \[
    \#(V(\mathcal{G}) \cap B_{X}(z_{n}, \epsilon n+\delta)) \le \alpha^{2\epsilon n}
    \]
    for all sufficiently large $n$.
   \end{itemize}
   Then, for $\epsilon > 0$, the subset $D_{\epsilon} \subset \partial X$ defined by 
   \[
   D_{\epsilon} = \bigg\{ \xi \in \partial X \colon \liminf_{r \to 0} \frac{\log \nu(B_{\rho}(\xi, r))}{\log r} \ge \frac{h}{l} - \epsilon \bigg\}
   \]
   is $\nu$-positive. Here $\nu$ denotes the harmonic measure $\bnd_{*}\mathbb{P}$. 
\end{lemma} 

In the proof, we consider the conditional measures associated to $\bnd \colon (X^{\Z_{+}}, \mathbb{P}) \to (\partial X, \nu)$. More explicitly, that is a family of probability measures $(\mathbb{P}^{\xi})_{\xi \in \partial X}$ on $X^{\Z_{+}}$ such that
  \[
  \mathbb{P} = \int_{\partial X} \mathbb{P}^{\xi} \, \, d\nu.
  \]

\begin{proof}
    It is enough to prove that for every $\epsilon>0$ there exists a $\nu$-positive set $F_{\epsilon}$ such that
    \[
    \liminf_{r \to 0} \frac{\log \nu(F_{\epsilon} \cap B_{\rho}(\xi, r))}{\log r} \ge \frac{h}{l} - \epsilon 
    \]
    for $\nu$-almost every $\xi \in \partial X$. Indeed, for such $F_{\epsilon}$, we have $\dim \nu \ge \dim \nu|_{F_{\epsilon}} \ge \frac{h}{l} - \epsilon$ by the definition of Hausdorff dimensions and Lemma \ref{Frostman}. Letting $\epsilon\to0$, we obtain $\dim\nu\ge h/l$. By Lemma \ref{Frostman}, this implies that $D_{\epsilon}$ is positive for every $\epsilon>0$. In the rest of the proof, we construct such $F_{\epsilon}$. Fix $0<\epsilon<1/2$ and define $A_{\epsilon, N}$ as follows:
    \[
    A_{\epsilon, N} = \{ x \in X^{\Z_{+}} \colon d_{X}(x_{n}, \xi_{x}(ln)) \le \epsilon n \, \, \text{and} \, \, \mathbb{P}(X_{n} = x_{n}) \le e^{-(h-\epsilon)n} \, \, \text{for every} \, \, n \ge N \}
    \]
    for $N \in \N$. Then there exists $N_{\epsilon}$ such that $\mathbb{P}(A_{\epsilon, N_{\epsilon}}) \ge 1-\epsilon$ by the first assumption. Let $A_{\epsilon} = A_{\epsilon, N_{\epsilon}}$ and $F_{\epsilon} = \{ \xi \, \colon \, \mathbb{P}^{\xi}(A_{\epsilon}) \ge \epsilon \}$. Then we can estimate the size of $F_{\epsilon}$ as follows: 
    \begin{align*}
    1- \epsilon
    &\le \mathbb{P}(A_{\epsilon}) = \int \mathbb{P}^{\xi}(A_{\epsilon}) \, d\nu(\xi) \\
    &= \int_{F_{\epsilon}} \, \mathbb{P}^{\xi}(A_{\epsilon}) \, d\nu(\xi) + \int_{F^{c}_{\epsilon}} \, \mathbb{P}^{\xi}(A_{\epsilon}) \, d\nu(\xi) \\
    &\le \nu(F_{\epsilon}) + \epsilon
    \end{align*}
    and hence $\nu(F_{\epsilon}) \ge 1-2\epsilon$. Let $N \ge N_{\epsilon}$ and $R > \max\{4\delta, R_{0}\}$, where $R_{0}$ is the constant in Lemma \ref{shadowball}. For $\mathbb{P}$-almost every $z \in A_{\epsilon, N}$, we define $y_{n}$ as $\xi_{z}(ln)$. Then, by Lemma 3.6 in \cite{MR3893268} (note that the assumption $R > 4\delta$ is used here), there exists $C > 0$ such that for all $n \ge N$ and $x \in A_{\epsilon}$, $x_{n}$ belongs to the ball $B_{X}(y_{n}, 2\epsilon n + C)$ if $\xi_{x} \in S(y_{n}, R)$. Therefore, we have 
    \begin{align*}
    \mathbb{P}(\xi_{x} \in F_{\epsilon} \cap S(y_{n}, R)) 
    &\le \mathbb{P}(A_{\epsilon} \cap \{ x_{n} \in B_{X}(y_{n}, 2\epsilon n + C) \}) + \mathbb{P}(A^{c}_{\epsilon} \cap \{ \xi_{x} \in F_{\epsilon} \cap S(y_{n}, R) \}) \\
    &\le \mathbb{P}(A_{\epsilon} \cap \{ x_{n} \in B_{X}(y_{n}, 2\epsilon n + C) \}) + (1-\epsilon)\nu(F_{\epsilon} \cap S(y_{n}, R))
    \end{align*}
    for every $n \ge N$.
    The last inequality follows from the construction of $F_{\epsilon}$. Then we have 
    \begin{align*}
    \epsilon \nu(F_{\epsilon} \cap S(y_{n}, R)) 
    &\le \mathbb{P}(A_{\epsilon} \cap \{ x_{n} \in B_{X}(y_{n}, 2\epsilon n + C) \}) \\
    &\le e^{-(h-\epsilon)n} \ \#(V(\mathcal{G}) \cap B_{X}(y_{n}, 2\epsilon n + C))
    \end{align*}
    by the construction of $A_{\epsilon}$.
    From this inequality, we obtain
    \begin{align*}
    \liminf_{n \to \infty} \frac{\log \nu(F_{\epsilon} \cap S(y_{n}, R))}{-ln}
    &\ge \liminf_{n \to \infty} \frac{-\log \epsilon -(h-\epsilon)n + \log \#(V(\mathcal{G}) \cap B_{X}(y_{n}, 2\epsilon n + C))}{-ln} \\
    &= \frac{h-\epsilon}{l} - \limsup_{n \to \infty} \frac{\log \#(V(\mathcal{G}) \cap B_{X}(y_{n}, 2\epsilon n + C))}{ln} \\
    &\ge \frac{h-\epsilon}{l} - \limsup_{n \to \infty}\frac{4 \epsilon \log \alpha}{l} = \frac{h-(1+4\log \alpha)\epsilon}{l}
    \end{align*}
    where the last inequality follows from the second assumption and the fact that $d_X(z_n,y_n)=o(n)$ (by Theorem \ref{geodesictracking}).
    Further, by applying Lemma \ref{shadowball} to $y_{n}$, which is on a geodesic ray from $o$ to $\xi_{z}$, we have 
    \[
    B_{\rho}(\xi_{z}, C^{-1}e^{-ln+R}) \subset S(y_{n}, R)
    \]
    and hence
    \begin{align*}
        \liminf_{r \to 0} \frac{\log \nu(F_{\epsilon} \cap B_{\rho}(\xi_{z}, r))}{\log r}
        &= \liminf_{n \to \infty} \frac{\log \nu(F_{\epsilon} \cap B_{\rho}(\xi_{z}, C^{-1}e^{-ln + R}))}{\log (C^{-1}e^{-ln + R})} \\
        &= \liminf_{n \to \infty} \frac{\log \nu(F_{\epsilon} \cap B_{\rho}(\xi_{z}, C^{-1}e^{-ln + R}))}{-ln} \\
        &\ge \liminf_{n \to \infty} \frac{\log \nu(F_{\epsilon} \cap S(y_{n}, R))}{-ln}
    \end{align*}
    since $C$ and $R$ are positive constants. 

For $\mathbb{P}$-almost every $z \in A_{\epsilon, N}$, we obtain 
    \begin{align*}
        \liminf_{r \to 0} \frac{\log \nu(F_{\epsilon} \cap B_{\rho}(\xi_{z}, r))}{\log r} 
        &\ge \liminf_{n \to \infty} \frac{\log \nu(F_{\epsilon} \cap S(y_{n}, R))}{-ln} \\
        &\ge \frac{h-(1+4\log \alpha)\epsilon}{l} 
    \end{align*}
    from the above estimates. 
    Since the sequence $(A_{\epsilon, N})_{N \ge N_{\epsilon}}$ gives an exhaustion of $X^{\Z_{+}}$ modulo null sets, we obtain the desired property after replacing $\epsilon$ by a sufficiently small multiple of $\epsilon$.
\end{proof}
Finally, we can prove the main theorem. 
\begin{proof}[Proof of Theorem \ref{mainthmgen}]
  Let $\epsilon$ be a positive number. By Theorem \ref{upperbound}, it is enough to show that 
  \[
  \liminf_{r \to 0} \frac{\log \nu_{\omega}(B_{\rho}(\xi, r))}{\log r} \ge \frac{h}{l} - \epsilon
  \]
  for $\eta$-almost every $(\omega, \xi) \in \Xi$.
  We define a subset $D_{\epsilon}$ of the ray bundle $\Xi$ by
  \[
  D_{{\epsilon}} = \{ (\omega, \xi) \in \Xi \colon \liminf_{r \to 0} \frac{\log \nu_{\omega}(B_{\rho}(\xi, r))}{\log r} \ge \frac{h}{l} - \epsilon \}.
  \]
  Note that the above condition for $(\omega, \xi)$ is invariant under the left $K$-action, and hence $D_{\epsilon}$ is well-defined.
  We want to show that this $D_{\epsilon}$ is $\eta$-conull. By Theorem \ref{speedentropy} and Lemma \ref{growthcontrol}, Lemma \ref{lowerboundlem} is applicable to $\mathcal{G}_{\omega}$ for $\mu_{o}$-almost every $\omega$, and hence $D_{\epsilon}$ is $\eta$-positive. Then, since the equivalence relation $\mathcal{R}_{o}^{\Xi}$ is ergodic with respect to $\eta$, it is enough to show that $D_{\epsilon}$ is $\mathcal{R}_{o}^{\Xi}$-invariant. Let $(\omega, \xi)$ be an element in $D_{\epsilon}$, and let $g\in G$ satisfy $gK\in\omega$. Let $\mathbb{P}_{\omega, gK}$ denote the distribution of sample paths of the simple random walk on $\mathcal{G}_{\omega}$ starting at $gK$. First, for any Borel subset $A \subset \partial X$, we have 
    \begin{align*}
    (g^{-1}_{*}\nu_{\omega})(A) 
    &= \mathbb{P}_{\omega}(\{ x \in X^{\Z_{+}} \colon \xi_{x} \in gA\}) \ge \mathbb{P}_{\omega}(\bigcup_{n \in \Z_{+}} \{x_{n} = gK \} \cap \{ x \in X^{\Z_{+}} \colon \xi_{x} \in gA\}) \\
    &= \mathbb{P}_{\omega}(\bigcup_{n \in \Z_{+}} \{x_{n} = gK \}) \mathbb{P}_{\omega, gK}(\{ x \in X^{\Z_{+}} \colon \xi_{x} \in gA\}) \\
    &=  \mathbb{P}_{\omega}(\bigcup_{n \in \Z_{+}} \{x_{n} = gK \}) \mathbb{P}_{g^{-1}\omega}(\{ x \in X^{\Z_{+}} \colon \xi_{x} \in A \})  
    = \mathbb{P}_{\omega}(\bigcup_{n \in \Z_{+}} \{x_{n} = gK \}) \nu_{g^{-1}\omega}(A),
    \end{align*}
    where we have used the strong Markov property in the second equality. The third equality follows from the equality $\mathbb{P}_{g^{-1}\omega} = g^{-1}_{*}\mathbb{P}_{\omega, gK}$. Note that $\mathbb{P}_{\omega}(\bigcup_{n \in \Z_{+}} \{x_{n} = gK\}) > 0$
    since $gK\in\omega$ and $\mathcal{G}_\omega$ is connected. Therefore, we have
  \begin{align*}
      \liminf_{r \to 0} \frac{\log\nu_{g^{-1}\omega}(B_{\rho}(g^{-1}\xi, r))}{\log r} 
      &\ge \liminf_{r \to 0} \frac{\log\nu_{g^{-1}\omega}(g^{-1}B_{\rho}(\xi, c_{g}r))}{\log r} \\
      &= \liminf_{r \to 0} \frac{\log\nu_{g^{-1}\omega}(g^{-1}B_{\rho}(\xi, c_{g}r))}{\log c_{g}r} \frac{\log c_{g}r}{\log r}\\
      &= \liminf_{r \to 0} \frac{\log\nu_{g^{-1}\omega}(g^{-1}B_{\rho}(\xi, r))}{\log r} \\
      &\ge \liminf_{r \to 0} \bigg(\frac{\log g^{-1}_{*}\nu_{\omega}(g^{-1}B_{\rho}(\xi, r))}{\log r} - \frac{\log \mathbb{P}_{\omega}(\bigcup_{n \in \Z_{+}} \{ x_{n} = gK \})}{\log r}\bigg) \\
      &= \liminf_{r \to 0} \frac{\log \nu_{\omega}(B_{\rho}(\xi, r))}{\log r} \ge \frac{h}{l} - \epsilon
  \end{align*}
  by the above estimate and Lemma \ref{gammaball}. This completes the proof.
\end{proof}

\bibliographystyle{amsalpha}
\bibliography{masterthesis}
\end{document}